%% file: Main_Petronio_2024_knots_on_spines.tex
\documentclass[11pt]{article}
\usepackage{latexsym,amsfonts,amsthm,amsmath,amscd,amssymb,color,url}
\usepackage{graphicx}

\newtheorem{lemma}{Lemma}[section]
\newtheorem{thm}[lemma]{Theorem}
\newtheorem{rem}[lemma]{Remark}
\newtheorem{prop}[lemma]{Proposition}

\newcommand{\dimo}[1]{\vspace{2pt}\noindent\textit{Proof of \ref{#1}}.\ }
\newcommand{\finedimo}{{\hfill\hbox{$\square$}\vspace{2pt}}}

\newcommand\matZ{{\mathbb{Z}}}

\renewcommand{\hbar}{{\overline{h}}}

\newfont{\Got}{eufm10 scaled 1200}

\newcommand{\compo}{\,{\scriptstyle\circ}\,}
\newcommand{\mycap} [1] {\caption{\footnotesize{#1}}}

\newcommand\calF{{\mathcal F}}

\newcommand{\figfatta}[3]{\centerline{#2}\mycap{#3}}

\newcommand{\intsigma}{\textrm{int}(\sigma)}

\begin{document}

\title{Diagrams of links and bands\\ on 3-manifold spines and flow-spines}

\author{Carlo~\textsc{Petronio}\thanks{Partially supported by INdAM through GNSAGA and by MUR (Italia) through the PRIN 2022 project ``Geometry and topology of manifolds''}}

\maketitle

\begin{abstract}\noindent
The Reidemeister theorem states that any link in $3$-space can be encoded by a diagram (a suitably decorated projection) on a plane,
and provides a finite set of combinatorial moves relating two diagrams of the same link up to isotopy.
In this note we replace $3$-space by any 3-manifold $M$ 
and we extend the Reidemeister theorem (definition of the decoration and description of the combinatorial moves) in four situations,
taking diagrams either of links or of bands (collections of cylinders and M\"obius strips),
either on an almost special spine or on a flow-spine of $M$.
This partially reproves and extends a result of  Brand, Burton, Dancso, He, Jackson and Licata.

\smallskip

\end{abstract}

\newcommand{\Flows}{\textsf{Flows}}
\newcommand{\Genflows}{\textsf{Flows}^0}
\newcommand{\Curves}{\textsf{Curves}(\Sigma)}
\newcommand{\Gencurves}{\textsf{Curves}(\Sigma,F)^0}
\newcommand{\Pairs}{\textsf{Pairs}}
\newcommand{\Genpairs}{\textsf{Pairs}^0}

\input{Introduction_Petronio_2024_knots_on_spines}

\input{Links_Petronio_2024_knots_on_spines}

\input{Ribbons_Petronio_2024_knots_on_spines}

\input{Flow_Links_Petronio_2024_knots_on_spines}

\input{Flow_Ribbons_Petronio_2024_knots_on_spines}


\vspace{.5cm}

\noindent
Dipartimento di Matematica, 
Universit\`a di Pisa\\
Largo Bruno Pontecorvo, 5 --
56127 PISA -- Italy\\
\texttt{petronio@dm.unipi.it}

\end{document}

%% file: Introduction_Petronio_2024_knots_on_spines.tex
\section*{Introduction}

In the recent paper~\cite{Ben:et:al}, Brand, Burton, Dancso, He, Jackson and Licata, extending ideas of 
Cromwell~\cite{Cromwell} and Dynnikov~\cite{Dynnikov99, Dynnikov06}, 
established a theory of crossingless diagrams of links on a 
\emph{trivalent} spine $\Sigma$ of a $3$-manifold $M$ (a slight generalization of Matveev's notion of \emph{special} spine~\cite{Matveev:book}).
In particular, after noting that any link has such a diagram, they proved the remarkable fact that there exist
local combinatorial moves relating two such diagrams of the same link up to isotopy. Even if the main
focus of~\cite{Ben:et:al} is on crossingless diagrams, its starting point is a
Reidemeister-type theorem for generic diagrams, namely the description of local combinatorial
moves relating two generic diagrams of the same link up to isotopy, which
extends the old foundational work
of Reidemeister~\cite{Reidemeister} and Alexander-Briggs~\cite{AlexanderBriggs}. 
In this note we reprove this result providing some details not fully made explicit in~\cite{Ben:et:al}, 
and we extend the theory in three directions.

\smallskip

First, we
enrich the structure of a diagram $D$ on $\Sigma$ of a link $L$ so to encode a band $B$ (a union of cylinders and M\"obius strips) having
$L$ as a core. Next, we consider the case where the spine $\Sigma$ has the extra structure of a flow-spine
(one that carries on $M$ a generic vector field with traversing orbits and concave tangency to the boundary), and we
define link diagrams on such a $\Sigma$. Last, we define band-diagrams on flow-spines. 
In all three cases we again provide finitely many local combinatorial moves translating at the level of diagrams the notion
of isotopy in $M$ of the corresponding objects. For the case where $M$ is an integer homology sphere,
we also specialize the result for band-diagrams on flow-spines to one for framed links, showing an analogue of what is known
for spherical diagrams of framed links in 3-space.

\smallskip

Links and their diagrams have been for a long time, and still are, central objects in topology; we address the reader again to~\cite{Ben:et:al}
for some recent references. Bands consisting only of cylinders can also be described as framed links, 
which are at the root of the theory of Dehn
surgery and of the Lickorish-Wallace theorem~\cite{Lick62, Wallace}, one of the most important tools at the basis of a variety 
of results in the theory of $3$-manifolds, see for instance~\cite{Eudaveetal19,
Futeretal22,
Lack19,
Mallick24,
Martelli21,
Savk2024}
and the references quoted therein for recent developments.
Diagrams of framed links are the object of Kirby calculus~\cite{Kirby78, FennRourke}, another fundamental result on which,
for instance, the formalization of Witten's~\cite{Witten}
quantum invariants of $3$-manifolds is founded~\cite{KL, RT}.
The Kirby calculus itself was also the object of rather recent refinements~\cite{Habiroetal17,
Martelli12}. For all these reasons we believe that the theories of link and band diagrams
developed in this paper might be the basis of further developments. We remark
however that bands that are not framed links do not seem to have attracted much attention.
It appears that M\"obius strips in $3$-manifolds have only been considered when
\emph{properly} embedded or immersed (see the foundational~\cite{Cannonetal76}
and the recent~\cite{Hughesetal20}).


%% file: Links_Petronio_2024_knots_on_spines.tex
\section{Link diagrams on almost special spines}\label{lins:section}
In this section we review the theory of almost special, trivalent and special spines of $3$-manifolds,
we define the diagram of a link on such a spine and we provide moves on diagrams embodying 
the notion of link isotopy.

\paragraph{Almost special, trivalent and special spines}
For the whole section we fix a compact and connected 3-dimensional manifold $M$ with non-empty boundary. Note that, if $N$ is a closed $3$-manifold,
taking $M$ to be $N$ minus an open ball we have that $M$ is well-defined, and the theory of links up to isotopy in $N$ is equivalent to that in $M$.
Following Matveev's theory of complexity (see~\cite{Matveev:book}) 
we define a $2$-dimensional polyhedron $\Sigma$ to be:
\begin{itemize}

\item \emph{Almost special}, if every point of $\Sigma$ has a neighbourhood of one of the types (2), (1) or (0) shown in 
Fig.~\ref{local-spine:fig} --- and $\Sigma^{(j)}$ will denote the subset of $\Sigma$ of points of type ($j$);
\begin{figure}
\figfatta{local-spine}
{\includegraphics[scale=0.5]{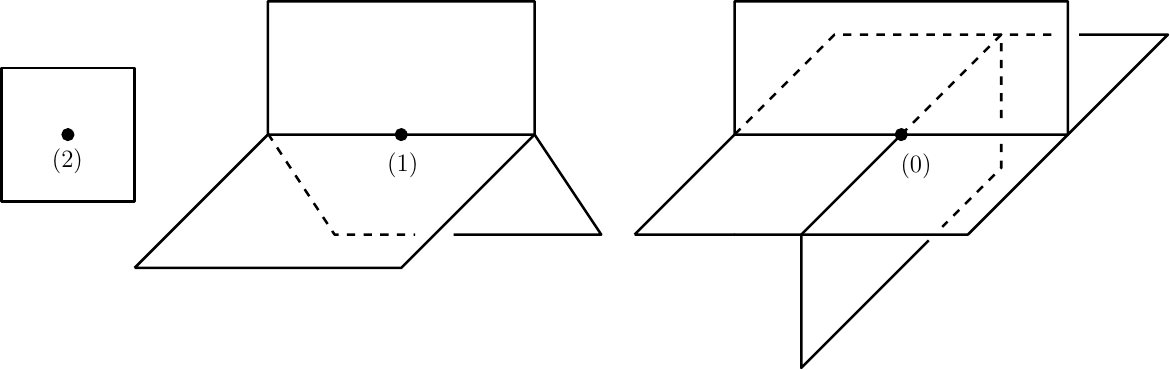}}
{Types of points in an almost special polyhedron.\label{local-spine:fig}}
\end{figure}

\item \emph{Special}, if it is almost special, $\Sigma^{(0)}$ is non-empty and the components of $\Sigma^{(2)}$ are open discs;

\item A \emph{spine} of $M$ if $\Sigma$ is tamely embedded in $M$ so that $M$ is a regular neighbourhood of $\Sigma$.
\end{itemize}

In~\cite{Ben:et:al} link diagrams are drawn on a polyhedron $\Sigma$ defined to be \emph{trivalent}, 
namely almost special with non-empty $\Sigma^{(1)}$. Note that a special polyhedron is trivalent, but in
this note we will only assume $\Sigma$ to be almost special. If  $\Sigma$ is an almost special spine of $M$ then there is a natural 
projection $\rho$ of $M$ onto $\Sigma$, as illustrated in Fig.~\ref{spine-projection:fig}
\begin{figure}
\figfatta{spine-projection}
{\includegraphics[scale=0.6]{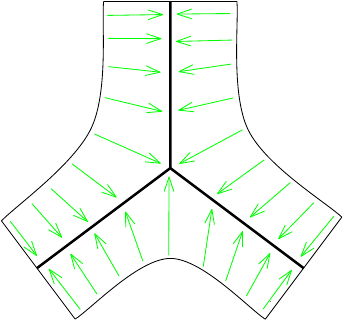}}
{The projection $\rho$ of $M$ onto an almost special spine $\Sigma$.\label{spine-projection:fig}}
\end{figure}
on a cross-section of a component of $\Sigma^{(1)}$ (the reader may easily visualize how the projection works near the points of $\Sigma^{(0)}$, called \emph{vertices}).
Conversely, $M$ can be viewed as a thickening of $\Sigma$, but (see~\cite{Matveev:book}):
\begin{itemize}
\item Not all almost special (or even special) polyhedra are thickenable;
\item The thickening (if any) of a special $\Sigma$ is uniquely defined, 
but this fails to be true if $\Sigma$ is only almost special, or even trivalent.
\end{itemize}

\bigskip

Until further notice, $\Sigma$ will be an almost special spine of $M$. 
We endow the components of $\Sigma^{(1)}$ (called \emph{edges}, even if some of them may be circles)
with a smooth structure, and the components of $\Sigma^{(2)}$ (called \emph{faces}) 
with a smooth structure compatible with that of $\Sigma^{(1)}$, so each face will be a 
smooth surface with boundary and corners at the vertices of $\Sigma$, possibly with self-adjacencies.

\paragraph{Link diagrams}
We will call \emph{link diagram} on $\Sigma$ a decorated subset $D$ of $\Sigma$ such that, for each face $\sigma$ of $\Sigma$:
\begin{itemize}
\item[(1)] $D\cap \sigma$ is a union of smooth immersed arcs and circles; 
\item[(2)] These arcs and circles cross themselves or each other transversely, in $\intsigma$, and at double points only;
\item[(3)] The circles lie in $\intsigma$ and the arcs meet $\partial \sigma$ transversely, at their endpoints, and not at the vertices of $\sigma$;
\end{itemize}
Moreover:
\begin{itemize}
\item Each point of $D\cap\Sigma^{(1)}$ on an edge $e$ of $\Sigma$ 
is the endpoint of precisely two portions of arc lying in locally different folds\footnote{These folds might globally belong to the same face of $\Sigma$,
so the two portions of arc might also be the endpoints of one and the same arc.}
of $\Sigma^{(2)}$ incident to $e$;
\item At each double point of $D$, on a face $\sigma$, the decoration consists of: 
\begin{itemize}
\item The choice of one of the two strands going through the double point (shown by a solid line in the figures, while the other strand is shown broken);
\item A transverse orientation for $\sigma$ in $M$ at the double point (shown by a thin arrow in the figures, where the convention is that locally $M$ is a 
neighbourhood of the portion of $\Sigma$ shown in $3$-space);
\end{itemize}
\item A decorated double point (called \emph{crossing})
is viewed to be the same if in the decoration both the choice of the strand and the transverse orientation are reversed, see
Fig.~\ref{crossing-dec-smooth:fig}.
\begin{figure}
\figfatta{crossing-dec-smooth}
{\includegraphics[scale=0.6]{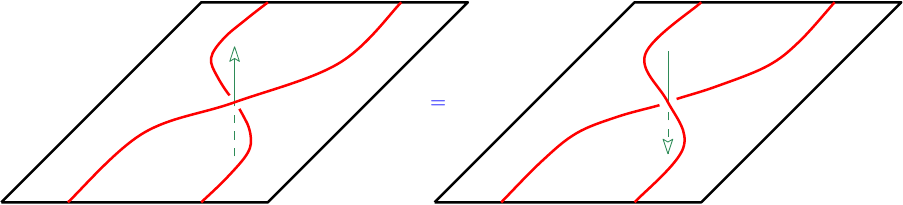}}
{Two equivalent decorations of a crossing.\label{crossing-dec-smooth:fig}}
\end{figure}
\end{itemize}

A link diagram $D$ on $\Sigma$ defines a link $L(D)$ in $M$ which coincides with $D$ except at the crossings, where the chosen strand is slightly pushed outside $\Sigma$
in the direction indicated by the transverse orientation. Of course $L(D)$ is well-defined up to isotopy.
Conversely, if $L$ is a link in $M$, up to small perturbation we can assume that the projection
$\rho:M\to\Sigma$ is generic when restricted to $L$, which implies that $\rho(L)$ can be decorated to become a link diagram $D$ such that $L(D)$ is isotopic to $L$.

\paragraph{Convention on moves}
In the sequel we will describe many combinatorial moves for link diagrams on $\Sigma$. Unlike~\cite{Ben:et:al} 
we will make the following convention: suppose that a move $\mu:D_0\to D_1$ is described  by local portions of diagrams $D_0$ and $D_1$ 
drawn in $3$-space, and that there is a symmetry $f$ of $3$-space such that $f(D_0)=D_0$; then the same symbol $\mu$
also encodes the move $D_0\to f(D_1)$. As examples of this convention, we show in Fig.~\ref{R1-move-smooth:fig}
\begin{figure}
\figfatta{R1-move-smooth}
{\includegraphics[scale=0.6]{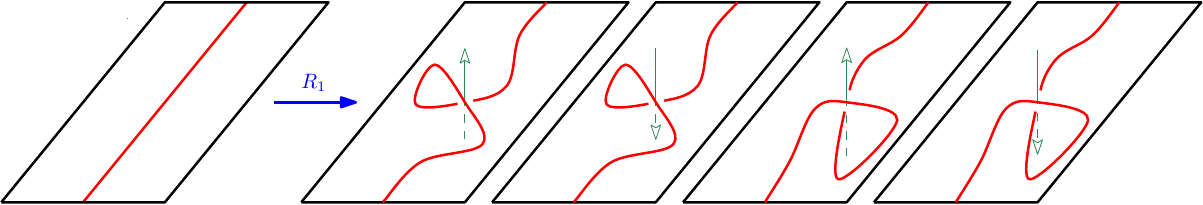}}
{These four versions of the Reidemeister move $R_1$ are viewed as the same move.\label{R1-move-smooth:fig}}
\end{figure}
the classical Reidemeister move $R_1$, and in Fig.~\ref{F2-move-smooth:fig}
\begin{figure}
\figfatta{F2-move-smooth}
{\includegraphics[scale=0.6]{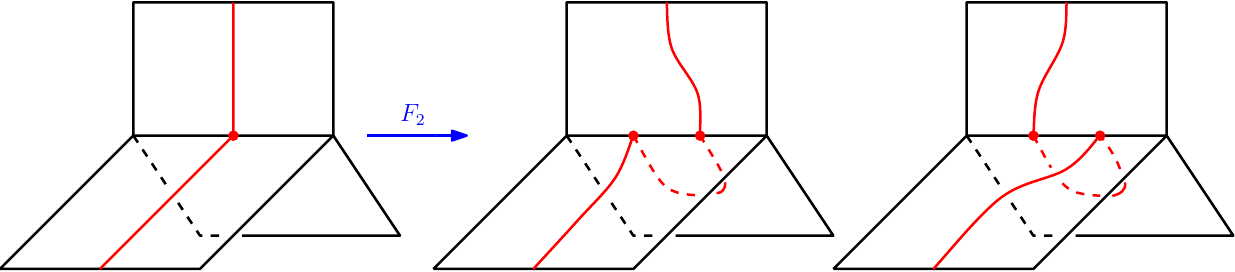}}
{These versions of the $F_2$ move are viewed as the same move.\label{F2-move-smooth:fig}}
\end{figure}
another move $F_2$ needed below and also used in~\cite{Ben:et:al} (where the two version are viewed as distinct and named FR and FL).

\paragraph{Moves relating diagrams of isotopic links}
In this paragraph we prove the following result:

\begin{thm}\label{main:link:thm}
Links in $M$ up to isotopy correspond bijectively to link diagrams on $\Sigma$ up to isotopy on $\Sigma$, 
the Reidemeister moves $R_1$, $R_2$ and $R_3$ away from $\Sigma^{(1)}$
and the moves $C$, $F_1$, $F_2$ and $V$ 
of Figg.~\ref{F2-move-smooth:fig} and~\ref{moves-C-F1-V:fig}.
\begin{figure}
\figfatta{moves-C-F1-V}
{\includegraphics[scale=0.6]{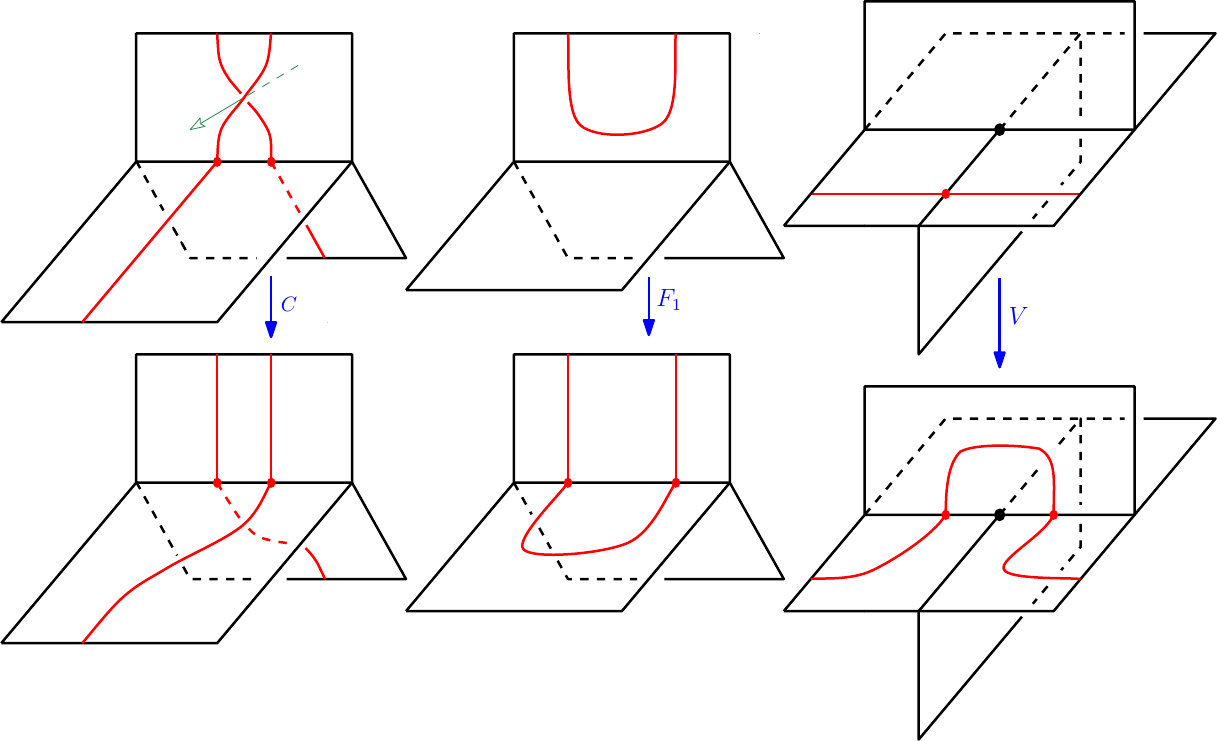}}
{The moves $C$, $F_1$ and $V$.\label{moves-C-F1-V:fig}}
\end{figure}
\end{thm}

While $R_1$ was shown in Fig.~\ref{R1-move-smooth:fig}, we do not depict $R_2$ and $R_3$, but we make the convention to give $R_2$ the direction in which the
number of vertices increases by $2$ (for $R_3$ the direction is immaterial).

\dimo{main:link:thm}
Of course every link has a diagram and the listed moves preserve the isotopy class of the link of a diagram.
So we must show that if $D_0$ and $D_1$ are link diagrams on $\Sigma$ and there exists an isotopy 
$(L_t)_{t\in[0,1]}$ between $L_0=L(D_0)$ and $L_1=L(D_1)$  then $D_0$ and $D_1$ are related by one of the listed moves.
Making $(L_t)_{t\in[0,1]}$ generic with respect to the projection $\rho:M\to\Sigma$,
there exist times $0<t_1<\ldots<t_N<1$ such that:
\begin{itemize}
\item For $t\neq t_j$ the projection
$\rho(L_t)$ can be turned into a diagram $D_t$ such that $L(D_t)$ is isotopic to $L_t$;
\item As $t$ varies in $[0,t_1)$, in $(t_N,1]$ and in each $(t_{j-1},t_{j})$, the diagram $D_t$ evolves by isotopy on $\Sigma$;
\item At $t=t_j$ the projection $\rho(L_t)$ undergoes an elementary catastrophe preventing it to be a diagram.
\end{itemize}
We must then show that these elementary catastrophes, namely first-order violations of 
the genericity conditions in the definition of a diagram, translate into the listed moves.
Setting $P=\rho(L_{t_j})$, the catastrophes are as follows:
\begin{itemize}
\item[(1)] $P\cap \sigma$ is not immersed in some face $\sigma$ at a point of $\intsigma$;
\item[(2.1)] $P\cap \sigma$ has a non-transverse double point in $\intsigma$ for some face $\sigma$;
\item[(2.2)] $P\cap \sigma$ has a transverse triple point in $\intsigma$ for some face $\sigma$;
\item[(2.3)] For some face $\sigma$, two\footnote{As in footnote 1, 
rather than two arcs these may be the terminal portions of one and the same arc.}
arcs $a,\ b$ of $P\cap \sigma$ cross each other transversely at a point $x$ of $\partial \sigma$
which is not a vertex of $\sigma$, and $a,\ b$ are transverse to $\partial \sigma$ at $x$;
\item[(3.1)] For some face $\sigma$, an arc or circle of $P\cap \sigma$ is internally 
tangent to $\partial\sigma$, but not at a vertex of $\sigma$;
\item[(3.2)] For some face $\sigma$, an arc of $P\cap \sigma$ is tangent to $\partial \sigma$ at one of its endpoints, but not at a vertex of $\sigma$;
\item[(3.3)] For some face $\sigma$, an arc of $P\cap \sigma$ is incident to $\partial \sigma$ at a vertex of $\sigma$ and it is transverse to both the edges of $\sigma$ incident to that vertex.
\end{itemize}
Of course the catastrophes (1), (2.1) and (2.2) translate into the Reidemeister moves $R_1$, $R_2$ and $R_3$, respectively.
For the catastrophe (2.3), we have to make a distinction: each of $a,\ b$ has a companion sharing the same endpoint
in one of the faces incident to $\sigma$, and these two companions can be in the same face (case (2.3.1), Fig.~\ref{case23:fig}-left)
\begin{figure}
\figfatta{case23}
{\includegraphics[scale=0.6]{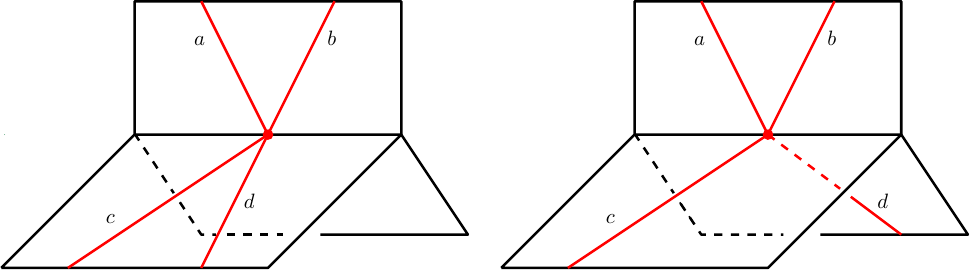}}
{Two local configurations for the catastrophe (2.3).\label{case23:fig}}
\end{figure}
or in different ones (case (2.3.2), Fig.~\ref{case23:fig}-right). Moreover, in case (2.3.1) we must further distinguish depending on whether 
the companions are $a$-$d$ and $b$-$c$, case (2.3.1.1) or $a$-$c$ and $b$-$d$, case (2.3.1.2).  Now case (2.3.1.1) is a catastrophe coming 
from a generic link isotopy precisely if it corresponds on diagrams to the move $S$ of Fig.~\ref{S-move-generated:fig}
\begin{figure}
\figfatta{S-move-generated}
{\includegraphics[scale=0.6]{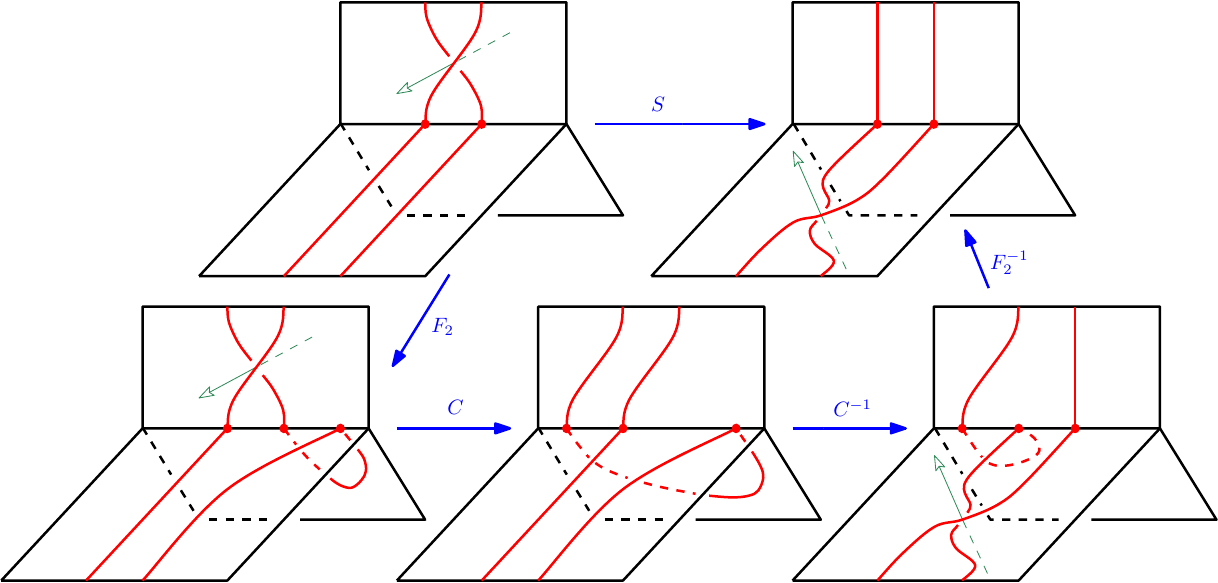}}
{The move $S$ and its realization via other moves.\label{S-move-generated:fig}}
\end{figure}
(where the letter $S$ evokes the \emph{sliding} of a crossing of the diagram past an edge of $\Sigma$).
However this move is shown already in Fig.~\ref{S-move-generated:fig} to be generated by the other ones in the statement, so it is not necessary.
Similarly, (2.3.1.2) gives the move of Fig.~\ref{move2312generated:fig}
\begin{figure}
\figfatta{move2312generated}
{\includegraphics[scale=0.6]{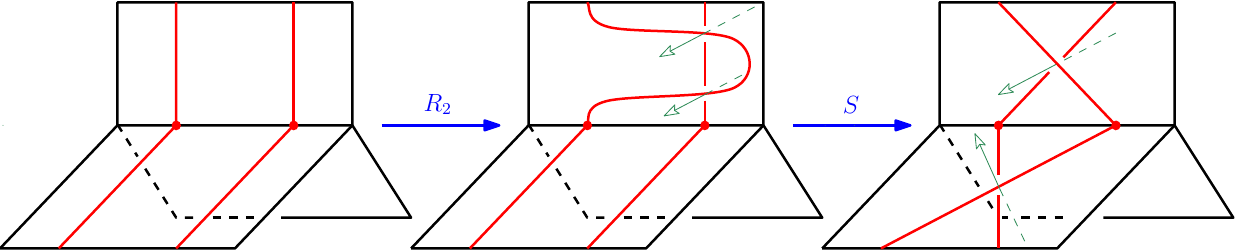}}
{The move coming from case (2.3.1.2) and its realization via other moves.\label{move2312generated:fig}}
\end{figure}
and the same figure shows it is also not necessary.

Turning to (2.3.2), we see that this case corresponds to a transition of projections as in  Fig.~\ref{pre-C-move-smooth:fig},
\begin{figure}
\figfatta{pre-C-move-smooth}
{\includegraphics[scale=0.6]{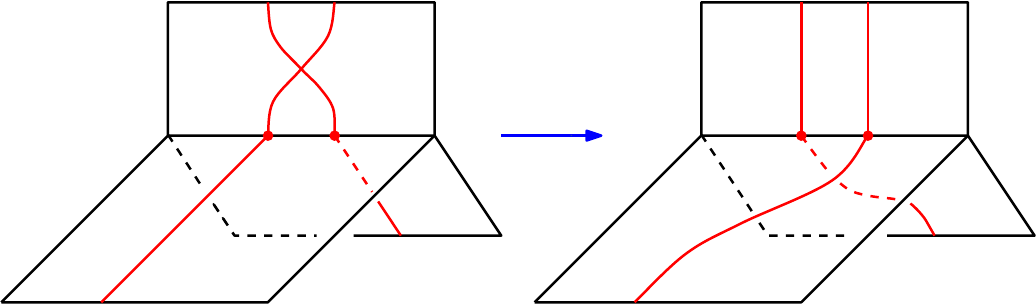}}
{The transition of projections giving the move $C$.\label{pre-C-move-smooth:fig}}
\end{figure}
and the only decoration of the double point for which this is coming from a link isotopy gives the move $C$ of the statement 
(where $C$ evokes the \emph{clearing} of a crossing).

Cases (3.1) and (3.2) are shown in Fig.~\ref{case3-1_and_2:fig}
\begin{figure}
\figfatta{case3-1_and_2}
{\includegraphics[scale=0.6]{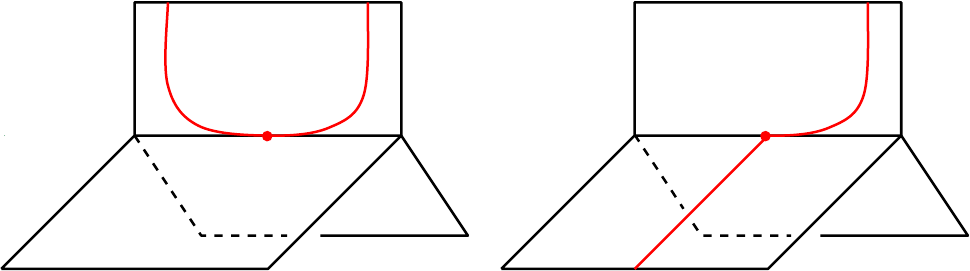}}
{Cases (3.1) and (3.2).\label{case3-1_and_2:fig}}
\end{figure}
and give rise respectively to the moves $F_1$ and $F_2$ of the statement
(where the letter $F$ evokes a \emph{finger} pushing a strand of link past an edge of $\Sigma$, as in ~\cite{Ben:et:al}). 
For case (3.2), note the following: before the catastophe, $P$ locally consists of two arcs contained in 
faces $\sigma,\sigma'$ of $\Sigma$ and transverse to $\partial\sigma,\ \partial\sigma'$; at the catastrophe
the second arc becomes tangent to $\partial\sigma'$; after the catastrophe the second arc is again transversal to $\partial\sigma'$ 
but the ends of the pre-existing arcs are distinct and joined by an arc that a priori can be contained
either in the third local face $\sigma'''$ or in $\sigma$.
The first case gives the move $F_2$, while the other one is actually not a generic catastrophe.

In case (3.3) again we need to make a distinction, because the arc in $\sigma$ that ends at a vertex of $\sigma$ must have a companion arc in one
of the other faces $\sigma'$ of $\Sigma$ incident to that vertex, and $\sigma'$ can be (locally) adjacent to $\sigma$ (case (3.3.1),  Fig.~\ref{case33-1_and_2:fig}-left)
\begin{figure}
\figfatta{case33-1_and_2}
{\includegraphics[scale=0.6]{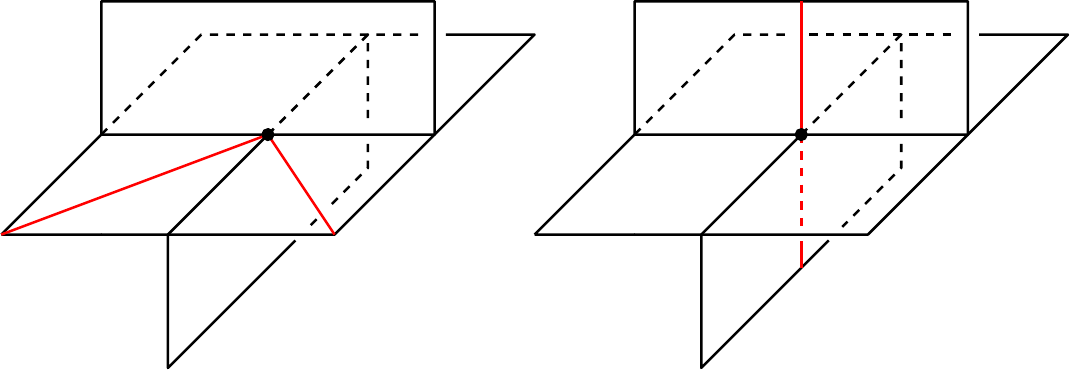}}
{Two local configurations for the catastrophe (3.3).\label{case33-1_and_2:fig}}
\end{figure}
or opposite to $\sigma$ (case (3.3.2),  Fig.~\ref{case33-1_and_2:fig}-right).
In case (3.3.1) we can assume that before the catastrophe locally the projection touches the interior of $\sigma$ and $\sigma'$ only,
and we must further distinguish according to whether after the catastrophe it touches the only other face of $\Sigma$ adjacent
to $\sigma$ and $\sigma'$ or the two faces of $\Sigma$ opposite to $\sigma$ and to $\sigma'$. This gives respectively the move $V$ of the statement
and the move of Fig.~\ref{move3312generated:fig},
\begin{figure}
\figfatta{move3312generated}
{\includegraphics[scale=0.6]{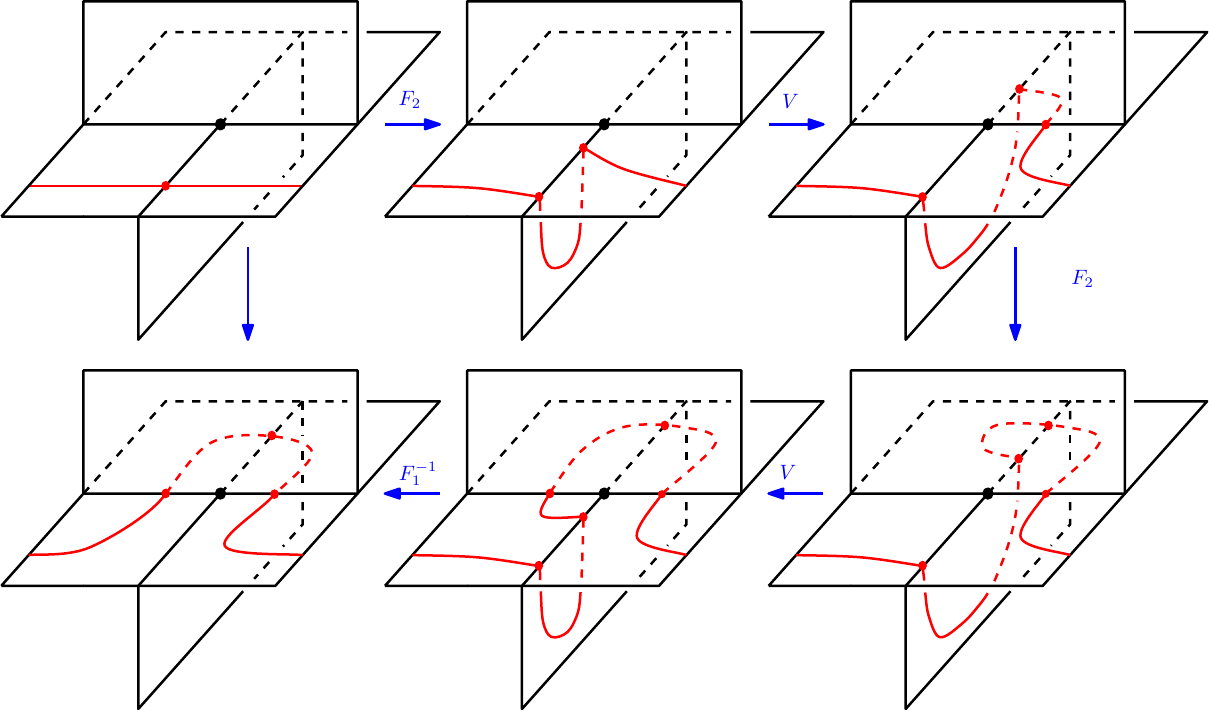}}
{A vertex move generated by other moves.\label{move3312generated:fig}}
\end{figure}
which is shown therein to be generated by the moves of the statement.
In case (3.3.2) the diagram will touch some faces $\sigma,\sigma',\sigma''$ before
the catastrophe and $\sigma,\sigma''',\sigma''$ after it, and we must distinguish according to whether $\sigma'$ and $\sigma'''$ are
adjacent, whence the move of Fig.~\ref{move3321generated:fig},
\begin{figure}
\figfatta{move3321generated}
{\includegraphics[scale=0.6]{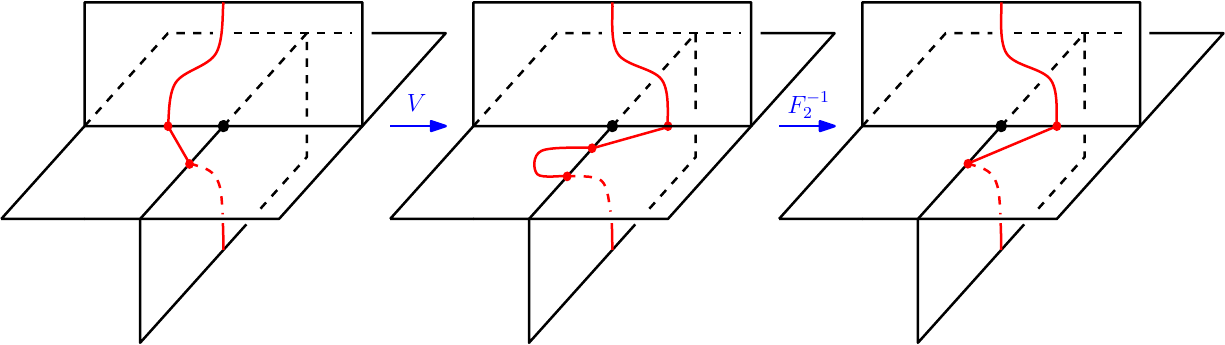}}
{Another vertex move generated by other moves.\label{move3321generated:fig}}
\end{figure}
shown therein to be generated by the moves of the statement, or opposite to each other, whence the move of 
Fig.~\ref{move3322:fig},
\begin{figure}
\figfatta{move3322}
{\includegraphics[scale=0.6]{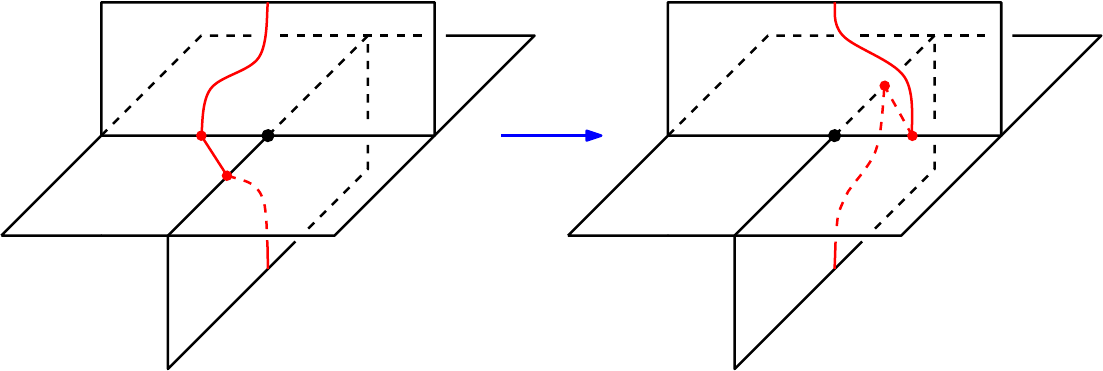}}
{One more vertex move.\label{move3322:fig}}
\end{figure}
which of course can be obtained by two moves as in Fig.~\ref{move3321generated:fig}.
This concludes the proof.
\finedimo

\paragraph{Alternative sets of moves}
Let us say that a set of moves for link diagrams on $\Sigma$ is \emph{complete} if it is sufficient to translate link isotopy.
Theorem~\ref{main:link:thm} states that $\{R_1$, $R_2$, $R_3$, $C$, $F_1$, $F_2$, $V\}$ is complete.
This set does not include the move $S$ of Fig.~\ref{S-move-generated:fig}, which may seem strange, since $S$ is the
most natural move coming from the interaction between diagram crossings and spine edges --- for sure, more natural than $C$.
So one might wonder whether $C$ can be replaced by $S$, but this is not the case:

\begin{rem}
\emph{The move $C$ is \emph{not} generated by $R_1$, $R_2$, $R_3$, $S$, $F_1$, $F_2$, $V$.
In fact, if there were a sequence of such moves generating $C$, this would apply also 
to the case where the crossing disappearing with $C$ involves two distinct components $K_1$ and $K_2$ of the link represented by the diagram.
But under these moves the parity of the number of crossings between $K_1$ and $K_2$ is preserved,
while this number can be odd if $\Sigma^{(1)}$ is non-empty.}
\end{rem}

We now remind that Theorem 3.2 in~\cite{Ben:et:al} states that 
$\{R_1$, $R_2, R_3$, $C_1$, $C_2$, $F_1$, $F_2,\ V\}$ is a complete set of moves,
for the moves $C_1$ and $C_2$ shown in Fig.~\ref{C1-C2-interchangeable:fig}, 
but the same figure shows the following, which implies that this set is actually redundant:
\begin{figure}
\figfatta{C1-C2-interchangeable}
{\includegraphics[scale=0.6]{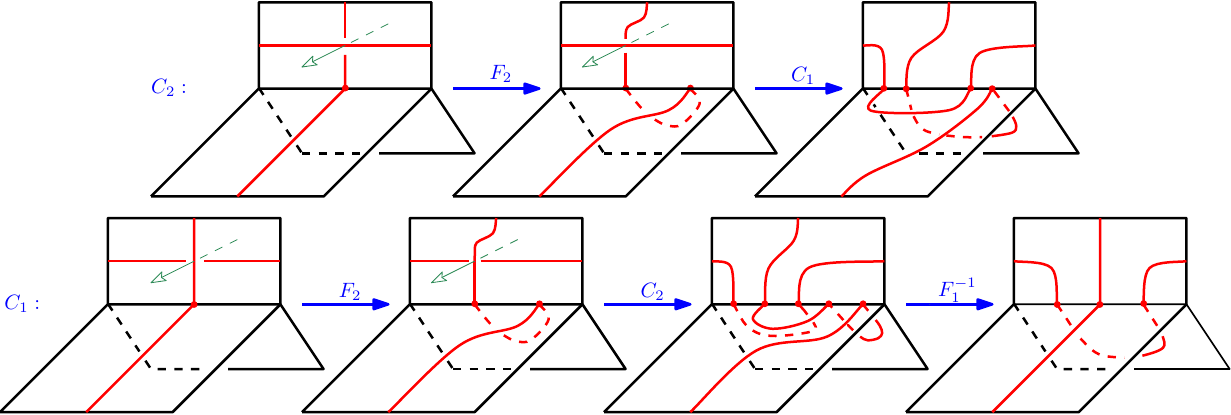}}
{Generating $C_1$ and $C_2$ from each other together with $F_1$ and $F_2$.\label{C1-C2-interchangeable:fig}}
\end{figure}

\begin{prop}
Dropping either $C_1$ or $C_2$ from 
$\{R_1$, $R_2$, $R_3$, $C_1$, $C_2$, $F_1$, $F_2, V\}$
one gets an equivalent set of moves.
\end{prop}

We next show that our Theorem~\ref{main:link:thm} and Theorem 3.2 in~\cite{Ben:et:al} can be deduced from each other, see
Fig.~\ref{C1-C-interchangeable:fig}:
\begin{figure}
\figfatta{C1-C-interchangeable}
{\includegraphics[scale=0.6]{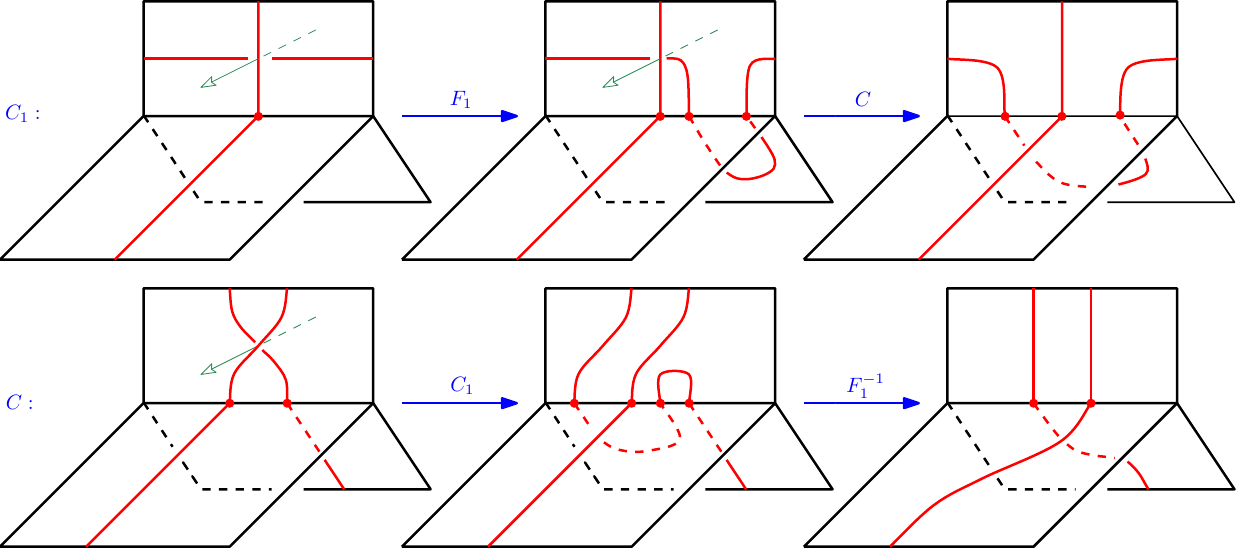}}
{Generating $C$ and $C_1$ from each other together with $F_1$.\label{C1-C-interchangeable:fig}}
\end{figure}

\begin{prop}
Replacing $C$ by $C_1$ in
$\{R_1$, $R_2$, $R_3$, $C_1$, $F_1$, $F_2, V\}$
one gets an equivalent set of moves.
\end{prop}

\paragraph{Relations with~\cite{Ben:et:al}}
We have stated above that our Theorem~\ref{main:link:thm} and Theorem 3.2 in~\cite{Ben:et:al} are equivalent to each other, 
however the settings of the two results differ as follows:
\begin{itemize}
\item As opposed to~\cite{Ben:et:al}, we do not 
assume that $M$ is oriented (but this assumption is actually not used in~\cite{Ben:et:al},
so this difference is immaterial);
\item In~\cite{Ben:et:al} the decoration of the double points of a diagram does not include the transverse orientation, 
which is indeed necessary for the diagram to unambiguously define a link; of course all the figures of the moves in~\cite{Ben:et:al} 
are correct if one makes the convention that the faces shown are transversely oriented in the direction exiting the page.
\end{itemize}

A more important remark concerns the proof of Theorem 3.2 in~\cite{Ben:et:al}. In fact, in the last two lines of page 1194 in~\cite{Ben:et:al},
the claim is made that any link $L$ can be isotoped to lie on $\partial N(\Sigma)$, namely on $\partial M$. However, this is not always the case:

\begin{prop}
There exist $M$ and $\Sigma$ and a knot $K$ in $M$ such that $K$ cannot be isotoped away from $\Sigma$.
\end{prop}

\begin{proof}
Let $N$ be $S^2\times S^1$.  Take a disc $\Delta$ in $S^2$ and let $P$ be the spine $(S^2\times\{\textrm{pt}\})\cup((\partial\Delta)\times S^1)$ of $N$ 
minus two balls, that we define as $M$.
Perturb $P$ to become a special spine $\Sigma$ of $M$ as suggested in Fig.~\ref{S2xS1_spine:fig}.
\begin{figure}
\figfatta{S2xS1_spine}
{\includegraphics[scale=0.6]{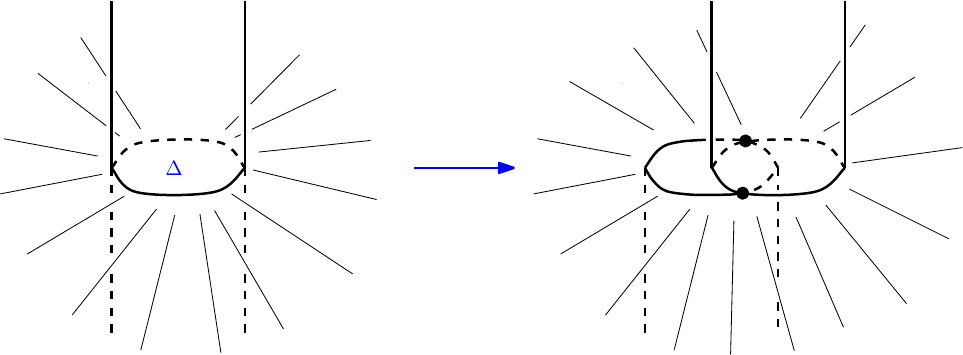}}
{A special spine of $S^2\times S^1$ minus two balls.\label{S2xS1_spine:fig}}
\end{figure}
Then $\Sigma$ still contains $S^2\times\{\textrm{pt}\}$, so $K=\{\textrm{pt}\}\times S^1$ cannot be isotoped away from $\Sigma$ in $N$ or $M$.
\end{proof}

This shows that the vertex moves of Figg.~\ref{move3312generated:fig},~\ref{move3321generated:fig} and~\ref{move3322:fig}
do occur along a generic isotopy of a link $L$ (or, at least, that they cannot be ruled out by the argument that $L$ can 
be isotoped on $\partial M$). One of the contributions of this paper is then to show that indeed they can be generated
by the vertex move $V$ and the other moves of the statement of our Theorem~\ref{main:link:thm}.

We also note that the move $F_2$ considered in~\cite{Ben:et:al} (named FL/FR there) does not come from the 
Reidemeister moves on $\partial M$ mentioned in the last two lines of page 1194 in~\cite{Ben:et:al}.
In fact, in the local configuration resulting from $F_2$ the link cannot be locally isotoped away from $\Sigma$.
The same applies to $C_2$.
We conclude by noting that the fact that the move $S$ of Fig.~\ref{S-move-generated:fig} is generated by the other moves 
is probably implicit in Fig.~5 of~\cite{Ben:et:al}.

\paragraph{Changing the spine}
So far, $\Sigma$ has been a fixed almost special spine of $M$.  However, if we restrict to a special $\Sigma$ 
with at least two vertices, then we can allow it to vary:

\begin{thm}\label{MP:link:thm}
Links in $M$ up to isotopy correspond to pairs $(\Sigma,D)$ 
where $\Sigma$ is a special spine of $M$ with at least two vertices embedded in $M$, and $D$ is a link diagram on $\Sigma$, 
up to the moves $R_1$, $R_2$, $R_3$, $C$, $F_1$, $F_2$, $V$ on $D$ for fixed $\Sigma$ and the \emph{MP} move of Fig.~\ref{MP-move:fig}
\begin{figure}
\figfatta{MP-move}
{\includegraphics[scale=0.6]{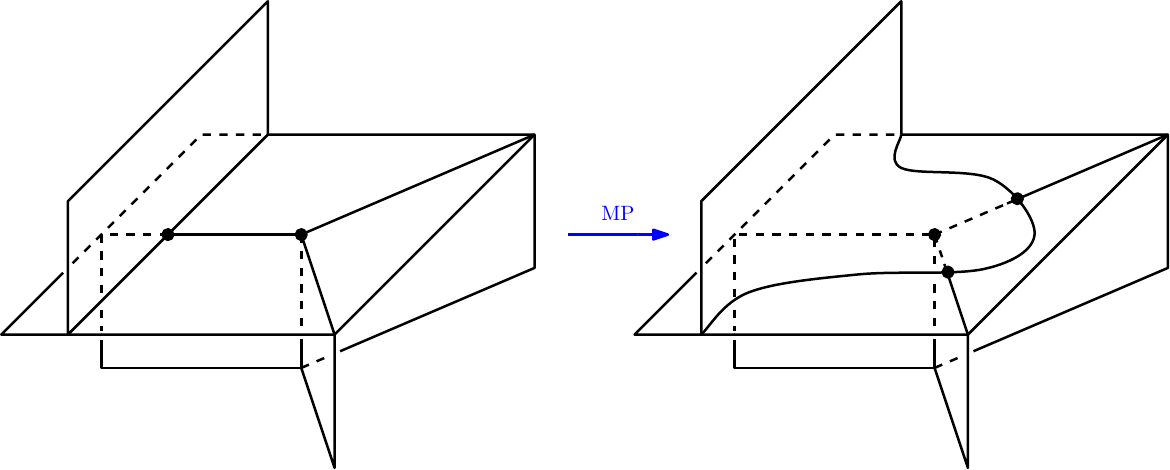}}
{The Matveev-Piergallini move.\label{MP-move:fig}}
\end{figure}
performed within $M$, with $D$ drawn away from the portion of $\Sigma$ affected by the move and unchanged by the move.
\end{thm}

\begin{rem}\emph{We must insist that $\Sigma$ is embedded in $M$ and that the \textrm{MP} move is performed within $M$ because
the same $\Sigma$ may have non-isotopic embeddings in $M$, which may lead one diagram on the the same abstract $\Sigma$ to define 
non-isotopic links in $M$.}
\end{rem}

\begin{proof}
Any $(\Sigma,D)$ as in the statement defines a link $L(\Sigma,D)$ up to isotopy, and for any link $L$ in $M$ and any $\Sigma$ there exists $D$ 
on $\Sigma$ such that $L(\Sigma,D)$ is $L$. Pairs related by moves as in the statement give isotopic links. Now suppose that $L(\Sigma,D)$ and
$L(\Sigma',D')$ are isotopic. By~\cite{Matveev:book} there exists a sequence $\Sigma=\Sigma_0\to\Sigma_1\to\ldots\to \Sigma_N=\Sigma'$
where each arrow is an \textrm{MP} move or an inverse \textrm{MP} move within $M$.
Note that an \textrm{MP} move (respectively, an inverse \textrm{MP} move) $\Sigma_j\to \Sigma_{j+1}$
causes the disappearance of an edge (respectively, of a triangular face) of $\Sigma_j$. 
Moreover:
\begin{itemize}
\item For every diagram $D_j$ on $\Sigma_j$ there exists another diagram $D'_j$ 
that avoids this edge (respectively, triangular face) such that $L(\Sigma_j,D_j)$ and $L(\Sigma_j,D'_j)$ are isotopic, because $L(\Sigma_j,D_j)$ 
can be isotoped away from a neighbourhood of the inverse image under $\rho$ of the edge (respectively, triangular face);
\item For such a $D'_j$, there is a link diagram $D_{j+1}$ on $\Sigma_{j+1}$ such that $L(\Sigma_j,D'_j)$ 
and $L(\Sigma_{j+1},D_{j+1})$ are isotopic, because $D'_j$ can be copied on $\Sigma_{j+1}$.
\end{itemize}
Using these facts one constructs a sequence
\begin{eqnarray*}
& (\Sigma,D)=(\Sigma_0,D_0)\to(\Sigma_0,D'_0)\to (\Sigma_1,D_1)\to(\Sigma_1,D'_1)\to(\Sigma_2,D_2)\to\ldots& \\
& \ldots \to (\Sigma_{N-1},D'_{N-1}) \to (\Sigma_N,D_N)=(\Sigma',D_N)\to(\Sigma',D') &
\end{eqnarray*}
all encoding isotopic links.  Hence all $(\Sigma_j,D_j)\to(\Sigma_j,D'_j)$ and $(\Sigma',D_N)\to(\Sigma',D')$
translate into moves $R_1$, $R_2$, $R_3$, $C$, $F_1$, $F_2$, $V$, while all $(\Sigma_j,D'_j)\to L(\Sigma_{j+1},D_{j+1})$
are moves as in Fig.~\ref{MP-move:fig}, whence the conclusion.
\end{proof}


%% file: Ribbons_Petronio_2024_knots_on_spines.tex
\section{Band diagrams}
If $L$ is a link in $M$, we call \emph{band with core $L$} a collection of 
cylinders and M\"obius strips embedded in $M$, one for each component of $L$ 
and having that component as a core.  A \emph{framed link with core $L$} is
a band with core $L$ with all cylinder components. We will denote henceforth 
by $U(L)$ a regular neighbourhood of $L$ in $M$, noting that each component of
$U(L)$ is either a solid torus or a solid Klein bottle (an orientable
or non-orientable $D^2$-bundle over $S^1$), and that a band with core $L$
can be realized up to isotopy as properly embedded in $U(L)$.

Let $\Sigma$ be an almost special spine of $M$. We call \emph{band-diagram} $\widetilde{D}$ on $\Sigma$ a link diagram $D$
endowed with the following extra structure:

\begin{itemize}
  \item Each portion of $D$ corresponding to a component $K$ of $L(D)$ such that $U(K)$ is a solid Klein bottle
  has a global label ``c'' or ``m'';
  \item On the rest of $D$ there are finitely many points $x$, called \emph{half-twist} points, which lie on $\Sigma^{(2)}$ and 
  are not crossings of $D$, with a decoration given by:
(1) An orientation for $D$ at $x$;
(2) A transverse orientation for $D$ in $\Sigma$ at $x$;
(3) A transverse orientation for $\Sigma$ in $M$ at $x$;
  \item Two half-twist point decorations are viewed as the same if they are obtained from each other 
  by simultaneous reversal of two of the orientations they consist of
  (so only two decorations of a point exist).
\end{itemize}

We will now  associate to such a $\widetilde{D}$ a band $B(\widetilde{D})$ with core $L(D)$, which requires a little preparation.
To begin we note that, for any knot $K$, any two homeomorphic proper bands with core $K$ in $U(K)$
differ at most by a move $\tau^n$ for $n\in\matZ$, where $\tau$ is the move that inserts a
full twist to $B$ in a portion of $U(K)$ locally trivialized as $D^1\times D^2$ 
where $B$ is initially $D^1\times D^1\times\{0\}$, see Fig.~\ref{band_double_twist:fig}.
\begin{figure}
\figfatta{band_double_twist}
{\includegraphics[scale=0.6]{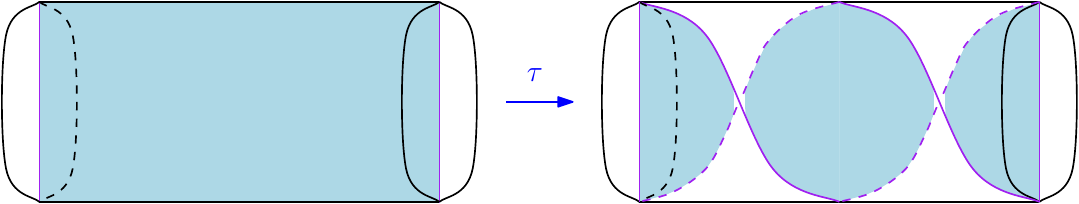}}
{A local move on a proper band  in a solid torus or Klein bottle.\label{band_double_twist:fig}}
\end{figure}
Note that $\tau$ is not quite well-defined, two possibly inequivalent versions exist, but all we say in the sequel holds true for either choice of $\tau$.

\begin{prop}\label{klein_isotopy_double_twist:prop}
If $K$ is a knot in $M$ and $U(K)$ is a solid Klein bottle, any two homeomorphic bands 
with core $K$ are properly isotopic within $U(K)$.
\end{prop}

\begin{proof}
Call the bands $B$ and $B'$.
We know $B'=\tau^n(B)$, but
Fig.~\ref{klein_isotopy_double_twist:fig} 
\begin{figure}
\figfatta{klein_isotopy_double_twist}
{\includegraphics[scale=0.6]{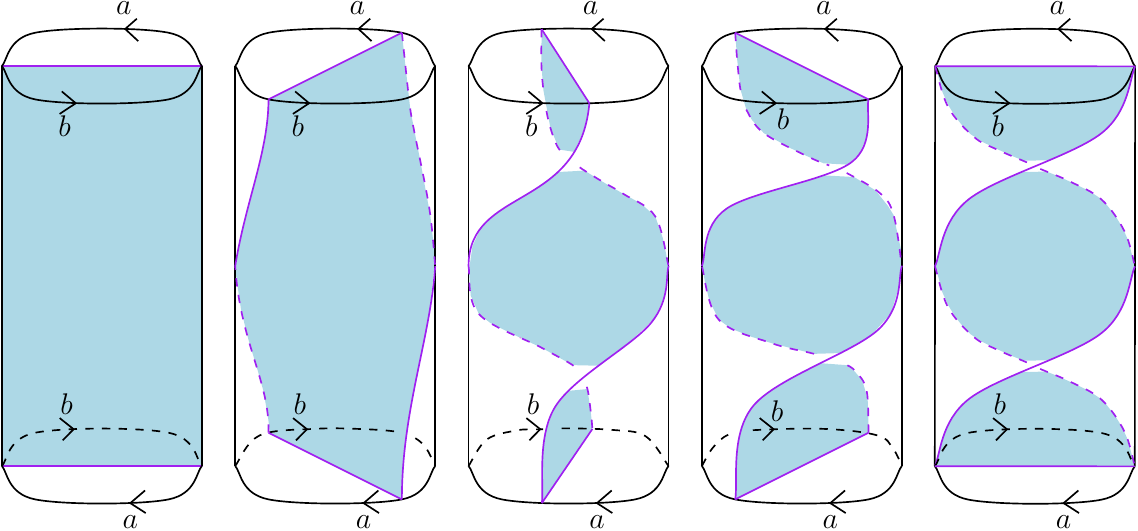}}
{Proper isotopy of a cylinder within a solid Klein bottle.\label{klein_isotopy_double_twist:fig}}
\end{figure}
shows that for $B$ a cylinder $\tau(B)$ is actually proper isotopic to $B$.
The same is true with a similar figure 
for $B$ a M\"obius strip, whence the conclusion.
\end{proof}

We can now define $B(\widetilde{D})$ for a band-diagram $\widetilde{D}$ on $\Sigma$. To a component $K$ of $L(D)$ such 
that $U(K)$ is a solid Klein bottle we associate a component of $B(\widetilde{D})$ which is either the only 
cylinder or the only M\"obius strip with core $K$, depending on whether the global label
of the corresponding portion of $D$ is ``c'' or ``m''. For the rest of $\widetilde{D}$, away from the crossings, 
the intersections with $\Sigma^{(1)}$ and the half-twist points,
$B(\widetilde{D})$ is just a regular neighbourhood of $D$ in $\Sigma$, while
Figg.~\ref{cross_sing_band:fig} and~\ref{band-dec:fig}
\begin{figure}
\figfatta{cross_sing_band}
{\includegraphics[scale=0.6]{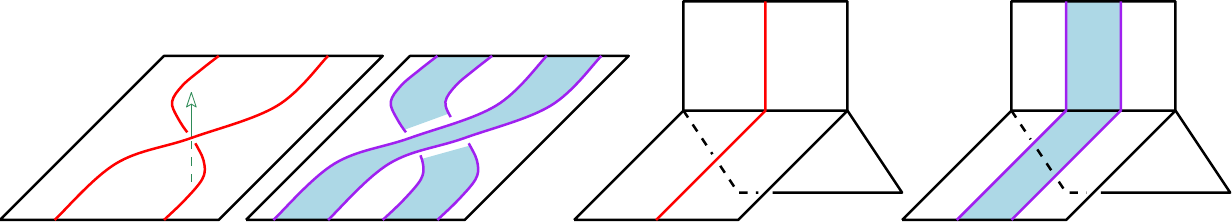}}
{The band $B(\widetilde{D})$ at a crossing of $D$ and at an intersection of $D$ with $\Sigma^{(1)}$.\label{cross_sing_band:fig}}
\end{figure}
show how to extend $B(\widetilde{D})$ near the special points. 
\begin{figure}
\figfatta{band-dec}
{\includegraphics[scale=0.6]{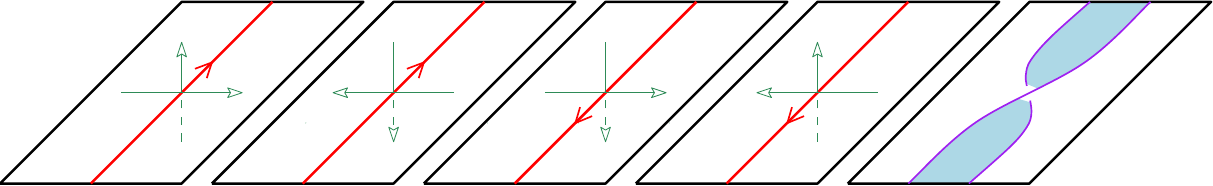}}
{Equivalent half-twist decorations of $\widetilde{D}$ and the associated band $B(\widetilde{D})$.\label{band-dec:fig}}
\end{figure}
Concerning the latter figure, the construction is well-defined
for equivalent decorations at a half-twist point $x$ because it can be intrinsically defined as follows (refer to the definition 
of the decoration for the numbering of the orientations): 
as one travels along (with respect to orientation 1) $D$ near $x$, the first (with respect to orientation 2) margin of $B(D)$  
passes over (with respect to orientation 3) the other margin.
The following is easy:

\begin{prop}\label{all:framed:arise}
Let $B$ be a band with core $L$ in $M$. Then there exists a band-diagram $\widetilde{D}$
such that $B(\widetilde{D})$ is isotopic to $B$.
Moreover, if $\Sigma$ is trivalent then one can choose $\widetilde{D}$ without half-twist points.
\end{prop}

\begin{proof}
Let $D$ be a diagram with $L(D)$ isotopic to $L$.
For a component $K$ of $L$ with $U(K)$ a solid Klein bottle, give the corresponding
portion of $D$ a global label ``c'' or ``m'' depending on whether the component of $B$
with core $K$ is a cylinder or a M\"obius strip. For a component $K$ of $L$ with $U(K)$ 
a solid torus, the corresponding portion of $D$ gives some band $\Gamma$ with core $K$.
Now if we add to $D$ one half-twist point we can change the topology of $\Gamma$, 
and if we add $2|n|$ half-twist points all of the same suitable type we can 
replace $\Gamma$ by $\tau^n(\Gamma)$. This easily implies that a suitably decorated 
version $\widetilde{D}$ of $D$ gives a $B(\widetilde{D})$ isotopic to $B$.
For the second statement, note that a half-twist can alternatively
be realized by performing a suitable $F_2$ move.
\end{proof}

\begin{rem}\emph{The second statement of the previous result implies that in an
orientable $M$, for a trivalent $\Sigma$, one can encode all bands using link diagrams.  However, inserting
the half-twist points seems inevitable when discussing the moves that relate two diagrams of the same band.}
\end{rem}

\paragraph{Band moves}
For band-diagrams on $\Sigma$ we consider the following moves:

\begin{itemize}
\item The moves $R_2,\ R_3,\ C,\ F_1,\ V$ (in which any global label, if present, is preserved, and no
half-twist point appears before or after the move);
\item The moves $\widetilde{R}_1$ and $\widetilde{F}_2$ shown in Fig.~\ref{R1_F2-move-framed:fig};
\begin{figure}
\figfatta{R1_F2-move-framed}
{\includegraphics[scale=0.6]{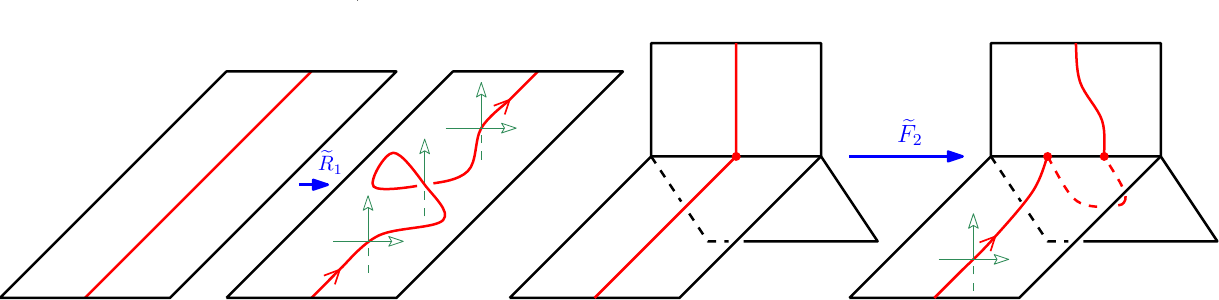}}
{Two moves on band-diagrams. In both, the involved portions of the diagram are not supposed to have a global label.
\label{R1_F2-move-framed:fig}}
\end{figure}
\item The move $A$ shown in Fig.~\ref{A-move-framed:fig}
\begin{figure}
\figfatta{A-move-framed}
{\includegraphics[scale=0.6]{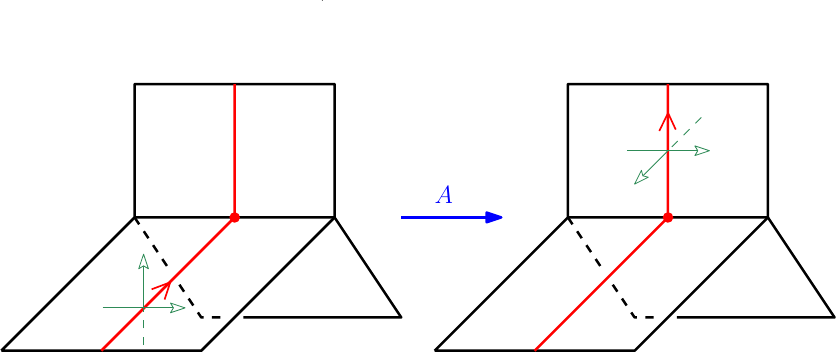}}
{Another move on band-diagrams.\label{A-move-framed:fig}}
\end{figure}
in which a half-twist point slides past an intersection of the diagram with $\Sigma^{(1)}$,
maintaining its type locally;
\item The move $B$ (not shown) in which a half-twist point slides past a crossing of the diagram 
maintaining its type locally;
\item The move $Z$ (not shown) in which two half-twist points of locally opposite type are created 
on an arc of the diagram.
\end{itemize}

One easily sees that these moves preserve the associated band up to
isotopy. Only for the moves $\widetilde{R}_1$ and $\widetilde{F}_2$ 
one needs to show that the local twisting created by the undecorated move
is annihilated by the insertion of half-twist points.
We will now show the following:

\begin{thm}\label{main:band:trivalent:thm}
Bands in $M$ up to isotopy correspond bijectively to 
band-diagrams on $\Sigma$ up to isotopy on $\Sigma$ and the 
moves $\widetilde{R}_1$, $ R_2$, $ R_3$, $ C$, $ F_1$, $\widetilde{F}_2$, $V$, $ A$, $ B$, $ Z$.
\end{thm}

The link analogue of this result (Theorem~\ref{main:link:thm}) was proved by analyzing 
the elementary catastrophes  that the projection to $\Sigma$ of the link undergoes along an isotopy.
Doing the same for bands seems however impossible, because even for a band isotopically moving within
$\Sigma$ the times at which the band is not associated to a diagram are not isolated (see for instance 
Fig.~\ref{band-F1:fig},
\begin{figure}
\figfatta{band-F1}
{\includegraphics[scale=0.6]{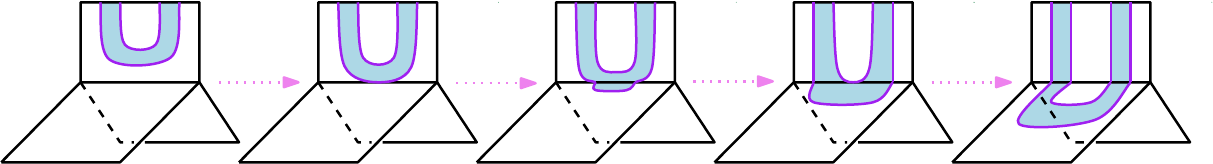}}
{An isotopy of a band contained in $\Sigma$.\label{band-F1:fig}}
\end{figure}
which corresponds to a move $F_1$). So we will need to use a different strategy, using Theorem~\ref{main:link:thm}
rather than mimicking its proof.
We begin with an easy fact which uses the move $\tau$ introduced before
Proposition~\ref{klein_isotopy_double_twist:prop}.

\begin{prop}\label{torus_double_twist:prop}
If $K$ is a knot in $M$ with $U(K)$ a solid torus and $B$ is a proper band in $U(K)$,
then $\{\tau^n(B):\ n\in\matZ\}$ are pairwise not properly isotopic in $U(K)$.
\end{prop}

\begin{proof}
With respect to a suitable longitude-meridian pair $(\lambda,\mu)$ on $\partial U(K)$,
we have that:
\begin{itemize}
  \item If $B$ is a cylinder, either component of $\partial \tau^n(B)$ represents $\lambda+n\cdot \mu$ in $H^1(\partial U(K))$;
  \item If $B$ is a M\"obius strip, $\partial \tau^n(B)$ represents $2\lambda+(2n+1)\cdot \mu$ in $H^1(\partial U(K))$.
\end{itemize}
This implies the conclusion.
\end{proof}

\dimo{main:band:trivalent:thm}
Every band has a diagram on $\Sigma$ and the listed moves preserve the isotopy class of the band of a diagram,
so we must prove that diagrams $\widetilde{D},\ \widetilde{D}'$ giving isotopic bands are related by the moves.
Let  $(B(t))_{t\in[0,1]}$ 
be an isotopy of bands with $B(0)=B(\widetilde{D}),\ B(1)=B(\widetilde{D}')$.
Let $L(t)$ be the core of $B(t)$ and assume up to small perturbation that $(L(t))_{t\in[0,1]}$ is generic with respect to the projection $\rho$ to $\Sigma$.
This means that there exist times $0=t_0<t_1<\ldots <t_N=1$ such that if $L_k=L(t_k)$ then
$\rho(L_k)$ is the support of a diagram $D_k$ of $L_k$, and
for $k=1,\ldots,N$ we have that 
$\rho(L_{k-1})$ and $\rho(L_k)$ are related 
by one of the elementary catastrophes discussed in the proof of Theorem~\ref{main:link:thm}.
A priori all the catastrophes can arise, not only those
giving the moves 
$R_1$, $R_2$, $R_3$, $C$, $F_1$, $F_2$, $V$ of the statement, but 
we can modify the band isotopy so that the projection of its core is modified as in
Figg.~\ref{S-move-generated:fig},~\ref{move2312generated:fig},~\ref{move3312generated:fig}, or~\ref{move3321generated:fig},
so we can assume that $D_{k-1}$ is transformed into $D_k$ by one of the moves
$R_1$, $R_2$, $R_3$, $C$, $F_1$, $F_2$, $V$ or the inverse of one of them.

We can now take for all $t$ a small tubular neighbourhood $U(t)$ of $L(t)$, noting that
$B(t)$ can be viewed as properly embedded in $U(t)$, 
and a diffeomorphism $\psi(t):U(0)\to U(t)$ such that $\psi(t)(L(0))=L(t)$, $\psi(t)(B(0))=B(t)$
and $\psi$ varies continuously with $t$. We then set 
$B_k=B(t_k)$, $U_k=U(t_k)$ and $\Psi_k=\psi(t_k)\compo\psi(t_{k-1})^{-1}:U_{k-1}\to U_{k}$, so that $\Psi_k(B_{k-1})=B_k$.
Next, we turn $D_k$ into a band-diagram $\widetilde{D}_k$ by adding global labels ``c'' or ``m'' and/or
half-twist points so that $B(\widetilde{D}_k)$ is properly isotopic to $B_k$ within $U_k$.
By construction $\Psi_k(B(\widetilde{D}_{k-1}))$ is properly isotopic to $B(\widetilde{D}_k)$ within $U_k$.
By applying to $\widetilde{D}_{k-1}$ the band version of the move that transforms 
$D_{k-1}$ into $D_k$ we get some other band version $\widetilde{D}'_k$ of $D_k$,
and $B(\widetilde{D}'_k)$ is properly isotopic to $\Psi_k(B(\widetilde{D}_{k-1}))$ within $U_k$, whence to $B(\widetilde{D}_k)$.
Note that $\widetilde{D}'_k$ and $\widetilde{D}_k$ both reduce to $D_k$ forgetting labels and half-twist points.
On each portion of $D_k$ corresponding to a component of $L_k$ with $U_k$ a solid Klein bottle 
we have that $\widetilde{D}'_k$ and $\widetilde{D}_k$ have the same global label.
On each portion of $D_k$ corresponding to a component of $L_k$ with $U_k$ a solid torus
we can apply the moves $A$ and $B$ to shift all the half-twist points of 
$\widetilde{D}'_k$ and $\widetilde{D}_k$ on one and the same arc of $D_k$. 
Locally we can define one type of half-twist to be positive, and the other one to be negative,
and Proposition~\ref{torus_double_twist:prop} implies that $\widetilde{D}'_k$ and $\widetilde{D}_k$
algebraically have the same number of half-twist points, which easily implies that they are related 
by some moves $Z$. We have shown that $\widetilde{D}_{k-1}$ and $\widetilde{D}_k$ are related by moves
as in the statement, but for $\widetilde{D}_0$ we can choose $\widetilde{D}$ and for $\widetilde{D}_N$ we can choose 
$\widetilde{D}'$, and the proof is complete.
\finedimo

As we did in the previous section, we could consider band-diagrams on a special spine $\Sigma$ of $M$ 
with at least two vertices, and allow $\Sigma$ to vary, getting a direct analogue of Theorem~\ref{MP:link:thm}.
We will not spell out this result explicitly.


%% file: Flow_Links_Petronio_2024_knots_on_spines.tex
\section{Link diagrams on flow-spines}
In this section we describe flow-spines of $3$-manifolds, slightly extending the theory founded in~\cite{Ishii1986, Ishii1992, BP-LNM}
(see also~\cite{Koda2007, Ishii2023, Ishii2024, Ishii2024bis} for unrelated recent developments of this theory).
We then define link diagrams on such spines and provide combinatorial moves on diagrams
embodying link isotopy.

\paragraph{Flow-spines}
Let $M$ be a fixed smooth $3$-manifold with non-empty boundary. 
We call \emph{generic concave traversing flow on} $M$ a partition $\calF$ of $M$ into oriented 
smooth arcs such that:
\begin{itemize}
\item There exists a smooth non-singular vector field on $M$ whose orbits are the elements of $\calF$;
\item Every $a\in\calF$ is tangent to $\partial M$ in at most two points, internal to $a$;
\item If $a\in\calF$ is tangent to $\partial M$ at $j$ internal points (so it is called a $j$-orbit) then $(M,\calF)$
in a neighbourhood of $a$ is diffeomorphic to one of the models of
Fig.~\ref{orbit_models:fig}.
\begin{figure}
\figfatta{orbit_models}
{\includegraphics[scale=0.6]{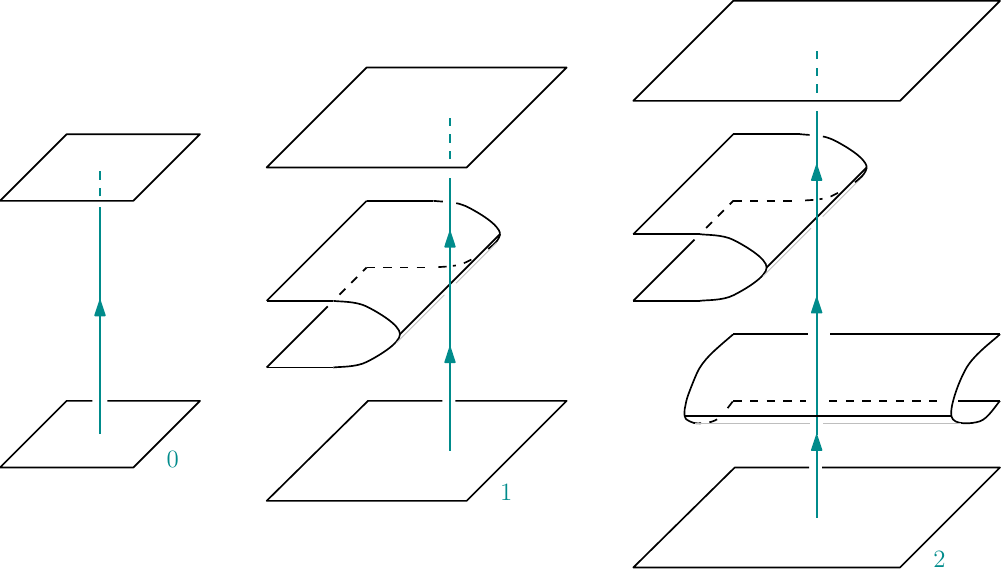}}
{The neighbourhood of a $j$-orbit $a$ in a generic concave traversing flow. 
Only $a$ is displayed, but also all the nearby orbits are those of an upward
vertical vector field.\label{orbit_models:fig}}
\end{figure}
\end{itemize}
We define $\partial_+M$ as the closure of the set of
final points of the orbits in $\calF$, and $\varphi:M\to\partial_+(M)$ as the map 
that associates to $x$ the final point of $a\in\calF$ if $x\in a$.
We then define the \emph{flow-spine} of $\calF$ as the following object $\Sigma$:
\begin{itemize}
\item Topologically, $\Sigma$ is $\partial_+(M)$ with every $x\in\partial(\partial_+M)$ identified to $\varphi(x)$, so
in particular $\Sigma$ is an almost special polyhedron;
\item $\Sigma$ is given a structure of branched surface and then pushed inside $M$ as shown in Fig.~\ref{flowspineconstruction:fig}
(which displays all the passages embedded in $M$ for a $1$-orbit, and only the final result for a $2$-orbit as in Fig.~\ref{orbit_models:fig}).
\begin{figure}
\figfatta{flowspineconstruction}
{\includegraphics[scale=0.6]{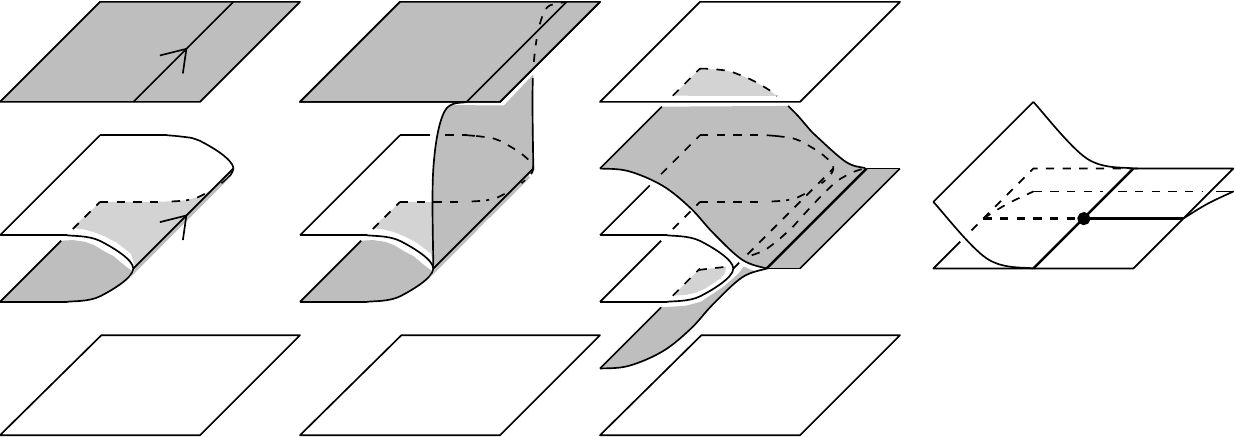}}
{Construction of the flow-spine of a flow.\label{flowspineconstruction:fig}}
\end{figure}
\end{itemize}

Intrinsically, a flow-spine of $M$ can be defined as an almost special polyhedron $\Sigma$ embedded in $M$ so that: 
\begin{itemize}
\item $M$ is a regular neighbourhood of $\Sigma$;
\item $\Sigma$ is endowed with a structure of
branched surface and transverse orientations for the components of $\Sigma^{(2)}$ that match along $\Sigma^{(1)}$ with
respect to the branching.
\end{itemize}
Of course from such a $\Sigma$ one can reconstruct a generic concave traversing flow on $M$.

\begin{rem}\label{changing:flow-spine:rem}
\emph{In the literature a flow-spine is usually also required to be special, but we can drop this restriction here.
Moreover, it was shown in~\cite{Ishii1986, BP-LNM} that if $N$ is closed and $v$ is a non-singular vector field on $N$
then there exists a closed ball $B$ in $N$ such that $(B,v)$ is diffeomorphic to $D^3$ with a constant vector field
and $v$ restricted to the complement $M$ of the interior of $B$ gives a generic concave traversing flow with 
associated special flow-spine $\Sigma$. Moreover in~\cite{Ishii1992, BP-LNM, Costantino} certain moves were introduced that allow 
to transform into each other two different such $\Sigma$'s for the same $(N,v)$, and also two different $\Sigma$'s for
the same $N$ and possibly different $v$'s.
}\end{rem}

\paragraph{Flow-link diagrams and moves}
Let $\Sigma$ be a fixed flow-spine of $M$ for a fixed flow $\calF$ as above. 
We call \emph{flow-diagram} on $\Sigma$ one which is a link diagram on $\Sigma$ as an almost special spine of $M$
(see Section~\ref{lins:section}) such that: (1) At the transverse intersections with $\Sigma^{(1)}$ the union of the two involved arcs of the
diagram is a smooth arc, so one of them lies in the locally one-sheeted portion of $\Sigma$ and the other one in the locally $2$-sheeted portion; 
(2) At the crossings the transverse orientation is always that intrinsic to $\Sigma$
(so it need not be shown).
 
A flow-diagram $D$ on $\Sigma$ defines a link $L(D)$ in $M$. To prove that every link arises as $L(D)$ for some $L$ we give some definitions.
Let $W$ be the union of the $1$- and $2$-orbits in $\calF$, and $A$ be the union of the $2$-orbits.
Note that $W\setminus A$ is a surface such that $\partial M\cap (W\setminus A)$ is $\partial (W\setminus A)$ disjoint union some
arcs and circles along which $\partial M$ and $W\setminus A$ are tangent to each other. Moreover the tangency points
to $\partial M$ split each $2$-orbit in $A$ into three segments, and $W\setminus A$ is triply incident
to the first and last of these segments and quadruply to the central one (see Fig.~\ref{thewall:fig}).
\begin{figure}
\figfatta{thewall}
{\includegraphics[scale=0.6]{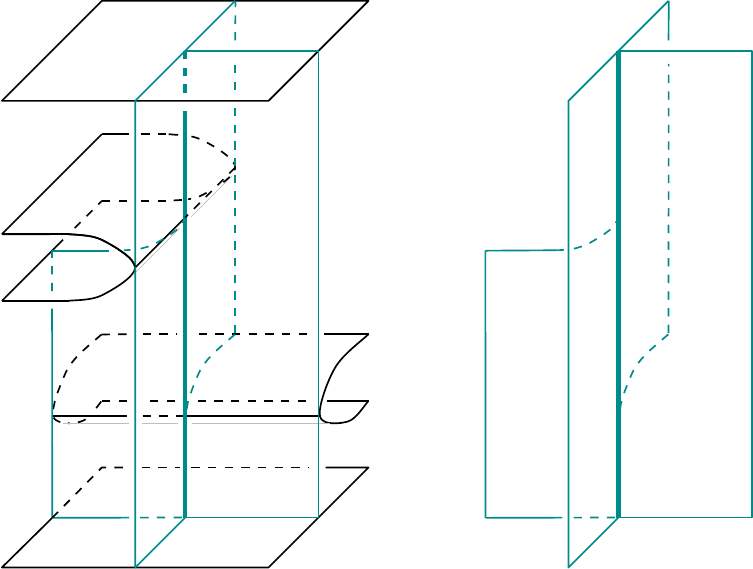}}
{$W$ near the $2$-orbit of Fig.~\ref{orbit_models:fig}, 
shown both embedded in $M$ and by itself.\label{thewall:fig}}
\end{figure}
We now say that
$L$ in $M$ is \emph{generic} with respect to $\calF$ if the following happens:
\begin{enumerate}
\item[(1)] $L$ is disjoint from $A$;
\item[(2)] $L$ intersects $W\setminus A$ transversely;
\item[(3)] $L$ intersects each $a\in\calF$ in at most two points;
\item[(4)] $L$ intersects each $a\in\calF$ transversely;
\item[(5)] If $L$ intersects $a\in\calF$ at two points $x_1,x_2$ then: (5.1) $a$ is a $0$-orbit; 
(5.2) if $\ell_1,\ell_2$ are small subarcs of $L$ centered at $x_1,x_2$ then $\varphi(\ell_1)$ and $\varphi(\ell_2)$ 
are transverse to each other on $\partial_+(M)$ at $\varphi(x_1)=\varphi(x_2)$.
\end{enumerate}
Note that any link $L$ becomes generic up to small perturbation. Moreover for a generic $L$ we can turn $\varphi(L)$
into a flow-diagram $D$ by specifying, in the situation of condition (5.2), that $\varphi(\ell_j)$ is the overarc
if $x_j$ is the last point on $a$ of $L\cap a$. For this $D$ we have that $L(D)$ is isotopic to $L$.
We now prove:

\begin{thm}\label{main:flow:thm}
Links in $M$ up to isotopy correspond bijectively to flow-diagrams on $\Sigma$ up to isotopy on $\Sigma$,
the Reidemeister moves $R_1$, $R_2$ and $R_3$ 
and the moves $F_1^{0,+},\ F_1^{0,-},\ F_1^{+,0},\ F_1^{-,0},\ C_0,\ S_+,\ S_-,\ V_+,\ V_0,\ V_-$ of 
Fig.~\ref{flowmoves:fig}.
\begin{figure}
\figfatta{flowmoves}
{\includegraphics[scale=0.6]{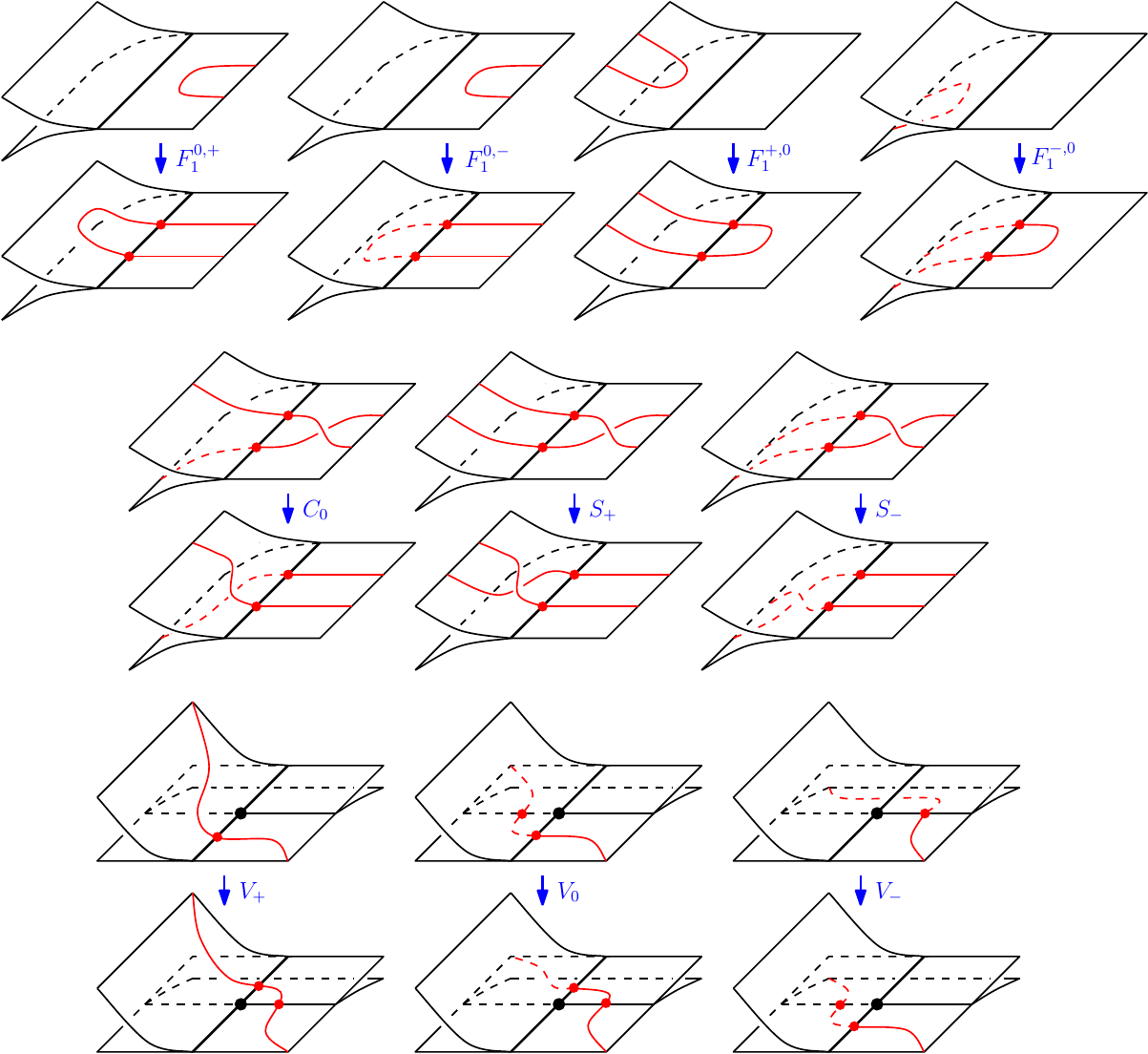}}
{Moves on flow-diagrams.\label{flowmoves:fig}}
\end{figure}
\end{thm}

\begin{proof}
We have seen above that any link has a flow-diagram, and of course the listed moves preserve the isotopy class of the link of a flow-diagram.
To conclude we must show that if $D_0$ and $D_1$ are flow-diagrams on $\Sigma$  and there exists an isotopy 
$(L_t)_{t\in[0,1]}$ between $L_0=L(D_0)$ and $L_1=L(D_1)$ in the interior of $M$ then
$D_0$ and $D_1$ are related by the listed moves. Assuming the isotopy is generic,
we must examine which elementary catastrophes can occur along it.

A catastrophe is a violation of one of the genericity conditions, and we can assume that only one occurs at each time $t$.
If (4) is violated we can assume that $L_t$ is tangent to some $0$-orbit $a\in\calF$
at one point, which gives the move $R_1$. If (5.2) is violated then we have $R_2$. If (3) is 
then we can assume that $L_t$ intersects thrice transversely some $0$-orbit $a\in\calF$
and that the $\varphi$-images of the intersection arcs are transverse on $\partial_+M$, and we have $R_3$.
If (2) is violated we can assume that $L_t$ is tangent to $W$ at a single point $x$ of a $1$-orbit $a$ and that
$L_t$ is transverse to $a$ at $x$; depending on whether $x$ lies on $a$ before or after the tangency point to $\partial M$ 
and on the side of $W$ to which $L_t$ locally lies we get the four moves $F_1^{*,*}$.
If (5.1) is violated we can assume that $a$ is a $1$-orbit, that $L_t$ is transverse to $W$ at $x_1,x_2$ and that
the $\varphi$-images of small arcs of $L_t$ containing $x_1,x_2$ are transverse to each other;
now $x_1,x_2$ can be both before the tangency point of $a$ to $\partial M$,
one before and one after, or both after, and correspondingly we get $S_-,\ C_0,\ S_+$.

The violations of (1) require a little more discussion. So, suppose a small arc $\ell\subset L_t$
intersects a $2$-orbit $a$ at a point $x$. First of all we can assume that $\ell$ is transverse to $a$ at $x$ 
and to the (3 or 4  --- see Fig.~\ref{thewall:fig}) folds of $W$ incident to $a$ at $x$.
We now need to distinguish two things:
\begin{itemize}
\item If $\sigma$ is the local face of $\Sigma$ at the vertex corresponding to $a$
that is locally one-sheeted along both its sides, then $\varphi(\ell)$ may or not contain an arc in $\sigma$;
\item $x$ can belong to the first, central or last portion into which the tangency points to $\partial M$ divide $a$.
\end{itemize}
If $\varphi(\ell)$ does contain an arc in $\sigma$, depending on the portion of $a$ to which $x$ belongs 
we get the moves $V_-,\ V_0,\ V_+$. If it does not, we get the moves $V'_-,\ V'_0,\ V'_+$ shown in
Fig.~\ref{generatedflowVs:fig}
\begin{figure}
\figfatta{generatedflowVs}
{\includegraphics[scale=0.6]{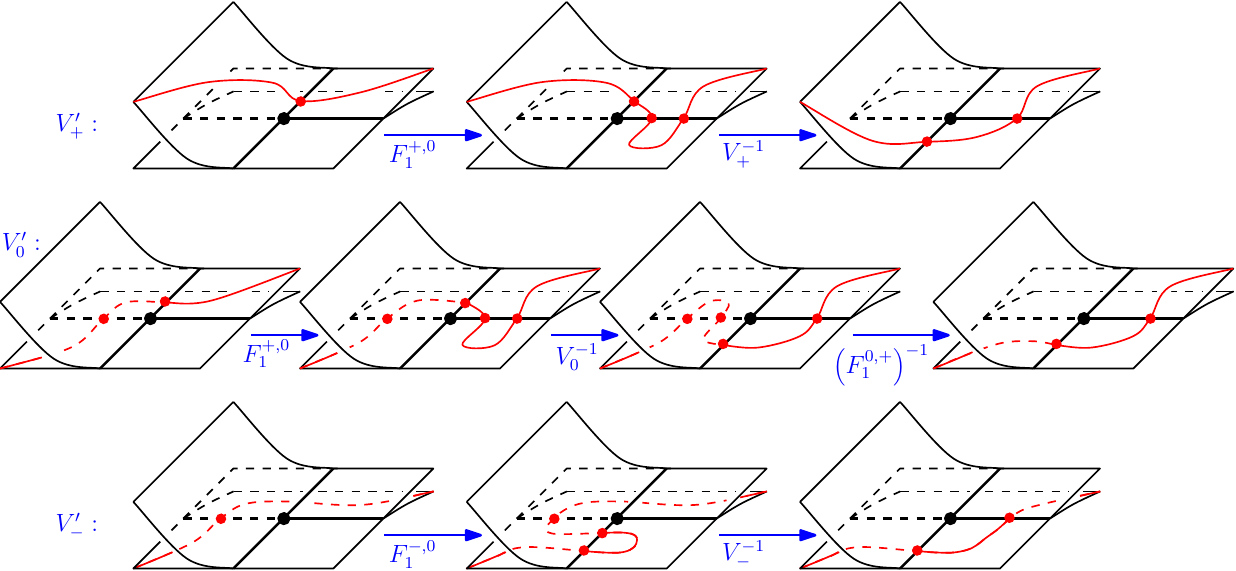}}
{Generating flow vertex moves by other moves.\label{generatedflowVs:fig}}
\end{figure}
to be generated by the $V_*$'s and the $F_1^{*,*}$'s.
\end{proof}

The moves of~\cite{Ishii1992, BP-LNM, Costantino}
referred to in Remark~\ref{changing:flow-spine:rem}, 
that allow to relate special flow-spines of the same manifold, are a lot
more complicated that the single MP-move of Fig.~\ref{MP-move:fig}.
One could nonetheless state an analogue of Theorem~\ref{MP:link:thm}
for link diagrams on varying special flow-spines, but we will 
not do this explicitly. The same will apply in the next section for
band-diagrams on special flow-spines.


%% file: Flow_Ribbons_Petronio_2024_knots_on_spines.tex
\section{Band diagrams on flow-spines}
In this section we provide a combinatorial presentation of bands in $M$ via diagrams on a flow-spine $\Sigma$ of $M$.
This will be just a combination of the contents of the previous two sections, but then we will specialize our result
to the case where $M$ is a homology sphere, getting a better presentation.

\paragraph{Flow-band-diagrams and moves}
A \emph{flow-band-diagram} $\widetilde D$ on $\Sigma$ is a flow-diagram $D$ with the addition
of global labels ``c'' or ``m'' on portions of $D$ giving components of $L(D)$ with
$U(L(D))$ a solid Klein bottle, and half-twist points on the rest of $D$, except that
now a transverse orientation is intrinsically defined on $\Sigma$, so a 
half-twist point is only decorated by a local orientation of $D$ and a local 
transverse orientation of $D$ on $\Sigma$, up to simultaneous reversal of both.

Note that the band move $\widetilde{R}_1$ of Fig.~\ref{R1_F2-move-framed:fig}
makes sense for flow-band-diagrams, and the same applies to the band moves $B$ and $Z$
and the flow-moves $F_1^{*,*},\ C_0,\ S_\pm,\ V_*$.
Instead, the band move $A$ of Fig.~\ref{A-move-framed:fig} gives rise to two flow-band-diagrams $A_\pm$,
because the local portion of diagram along which the half-twist point slides past $\Sigma^{(1)}$ may
have an arc on the upper or lower sheet of the locally two-sheeted portion of $\Sigma$.
For all these flow-band versions of the moves $\widetilde{R}_1$, $A_\pm$, $B$ and $Z$ one should
insist that the intrinsic transverse orientation of $\Sigma$ is used at half-twist points.

\begin{thm}\label{main:band:flow:thm}
Bands in $M$ up to isotopy correspond bijectively to 
flow-band-diagrams on $\Sigma$ up to isotopy on $\Sigma$ and the 
moves $\widetilde{R}_1$,  $R_2$, $R_3$,
$F_1^{*,*},\ C_0$,\ $S_\pm$,\ $V_*$, $A_\pm$, $ B$, $ Z$.
\end{thm}

\begin{proof}
As always, the only non-trivial point is to show that flow-band-diagrams defining isotopic bands
are related by moves. This is achieved by the same strategy used in the proof of Theorem~\ref{main:band:trivalent:thm}:
up to small perturbation and the local modifications of Fig.~\ref{generatedflowVs:fig} one can assume that
along the isotopy the core of the bands is the link defined by a flow-diagram on $\Sigma$,
except at finitely many times where a move $D_{k-1}\to D_k$ of type $R_1$,  $R_2$, $R_3$,
$F_1^{*,*},\ C_0,\ S_\pm,\ V_*$ takes place.
Then we insert half-twist points on each $D_k$ to get a flow-band $\widetilde{D}_k$
that defines the appropriate band, we define $\widetilde{D}'_k$ by replacing
$D_{k-1}\to D_k$ with its band version $\widetilde{D}_{k-1}\to \widetilde{D}'_k$ and
we prove that $\widetilde{D}'_k$ and $\widetilde{D}_k$ are related by the moves $A_\pm$, $ B$, $ Z$.
\end{proof}

\paragraph{Framed links in a homology sphere}
Suppose now that $M$ is an integer homology $3$-sphere $N$ minus some open balls.
Let $\Sigma$ be a flow-spine of $M$ and let $D$ be a flow-diagram on $\Sigma$,
viewed as a band-flow-diagram without half-twist points. 
Since $M$ is orientable and $\Sigma$ is transversely oriented, 
we see that $B(D)$ consists of cylinders, \emph{i.e.}, it is a framing on $L(D)$.

\begin{thm}\label{homology:sphere:thm}
Framed links in $M$ correspond bijectively to flow-diagrams on $\Sigma$ up to the moves
$R_2$, $R_3$, $F_1^{*,*},\ C_0,\ S_\pm,\ V_*$ and $U$ of Fig.~\ref{Umove:fig}.
\begin{figure}
\figfatta{Umove}
{\includegraphics[scale=0.6]{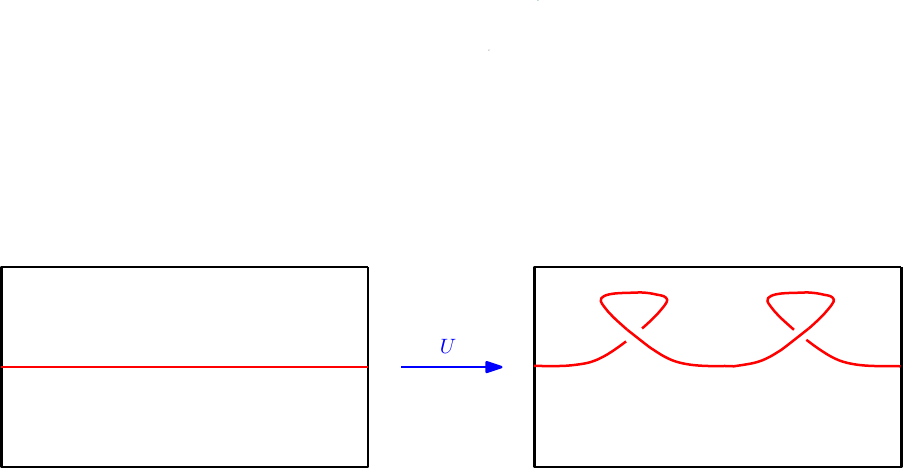}}
{A move on flow-diagrams. Sliding one of the created curls past the other one, which is achieved by moves $R_2$ and $R_3$, 
gives the same move with switched crossings.\label{Umove:fig}}
\end{figure}
\end{thm}

\begin{proof}
Given a framed link in $M$ and a flow-diagram representing its core, we can perform moves $R_1$ on this
diagram to get one that represents the framed link. Of course the listed moves do not change the isotopy
class of the framed link defined by a flow-diagram, so we must prove that flow-diagrams $D,\ D'$ with 
isotopic $B(D),\ B(D')$ are related by the moves.
To do this it will be useful to fix an orientation of $M$ (and hence of $\Sigma^{(2)}$, because
it has an intrinsic transverse orientation in $M$).

Since in particular $L(D)$ and $L(D')$ are isotopic, 
by Theorem~\ref{main:flow:thm} there exists a sequence of moves
$D=D_0\to D_1\to\ldots\to D_N=D'$ with each $D_{k-1}\to D_k$ of type
$R_1$,  $R_2$, $R_3$, $F_1^{*,*}$, $C_0$, $S_\pm$, or $V_*$. 
We now construct diagrams $\overline{D}_k$ and moves $\overline{D}_{k-1}\to\overline{D}_k$
where $\overline{D}_0=D_0$, each $\overline{D}_k$ differs from $D_k$ for the presence of some
\emph{frozen curls}, namely curls drawn inside small shadowed discs contained in the components of $\Sigma^{(2)}$ 
(see the pictures below), and each move $\overline{D}_{k-1}\to\overline{D}_k$ is realized as a combination of 
$R_2$, $R_3$, $F_1^{*,*},\ C_0,\ S_\pm,\ V_*$, so that $B(\overline{D}_{k-1})$ and $B(\overline{D}_k)$ are isotopic in $M$.
The construction takes place inductively as follows once $\overline{D}_{k-1}$ is constructed:
\begin{itemize}
\item If $D_{k-1}\to D_k$ is a positive $R_1$ move we can assume it takes place 
in $\overline{D}_{k-1}$ away from the existing frozen curls; 
then $\overline{D}_{k-1}\to\overline{D}_k$ is the move shown in Fig.~\ref{frozen-positive-R1:fig}:
\begin{figure}
\figfatta{frozen-positive-R1}
{\includegraphics[scale=0.6]{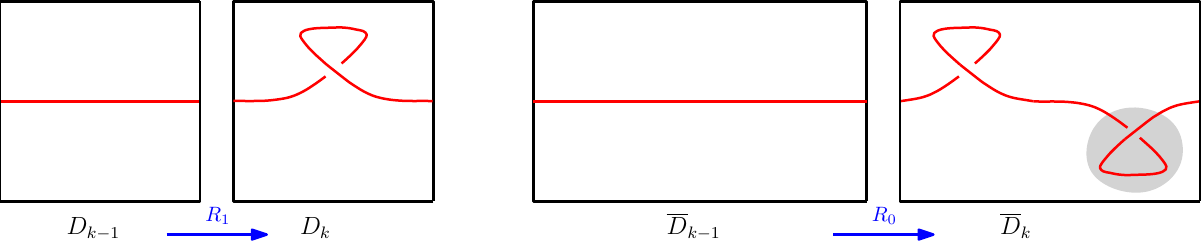}}
{Replacing a positive $R_1$ move by a move $R_0$.\label{frozen-positive-R1:fig}}
\end{figure}
the existing frozen curls of $\overline{D}_{k-1}$ persist in $\overline{D}_k$ and two positive $R_1$ moves 
are performed, the original one and one creating a curl which is declared frozen. For later purpose we call this move $R_0$ 
and recall the easy fact that it is realized by $R_2$'s and $R_3$'s;
\item If $D_{k-1}\to D_k$ is a move $R_1^{-1}$ then $\overline{D}_k$ coincides with $\overline{D}_{k-1}$, except that 
the the curl that $R_1^{-1}$ aimed at destroying is kept frozen, see Fig.~\ref{frozen-negative-R1:fig};
\begin{figure}
\figfatta{frozen-negative-R1}
{\includegraphics[scale=0.6]{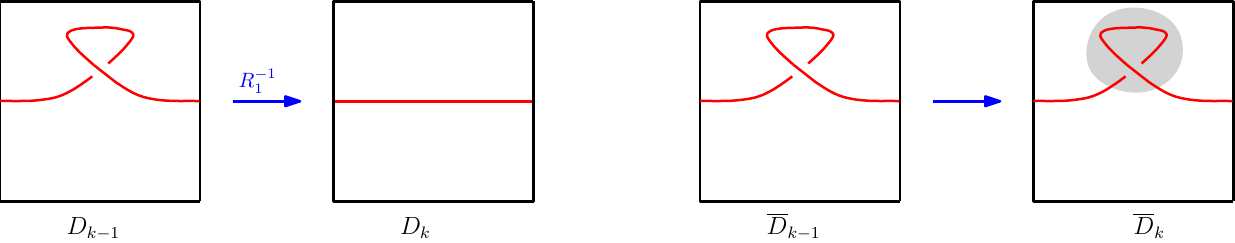}}
{Replacing an $R_1^{-1}$ move.\label{frozen-negative-R1:fig}}
\end{figure}
\item If $D_{k-1}\to D_k$ is not $R_1^{\pm1}$ then $\overline{D}_{k-1}\to\overline{D}_k$
is a sequence of moves $R_2$, $R_3$, $F_1^{*,*},\ C_0,\ S_\pm,\ V_*$ necessary to realize 
$D_{k-1}\to D_k$ by first sliding the frozen curls in $\overline{D}_{k-1}$
away from the area where $D_{k-1}\to D_k$ takes place and then performing it. Two 
examples are shown in Fig.~\ref{sliding-frozen-curls:fig}
\begin{figure}
\figfatta{sliding-frozen-curls}
{\includegraphics[scale=0.6]{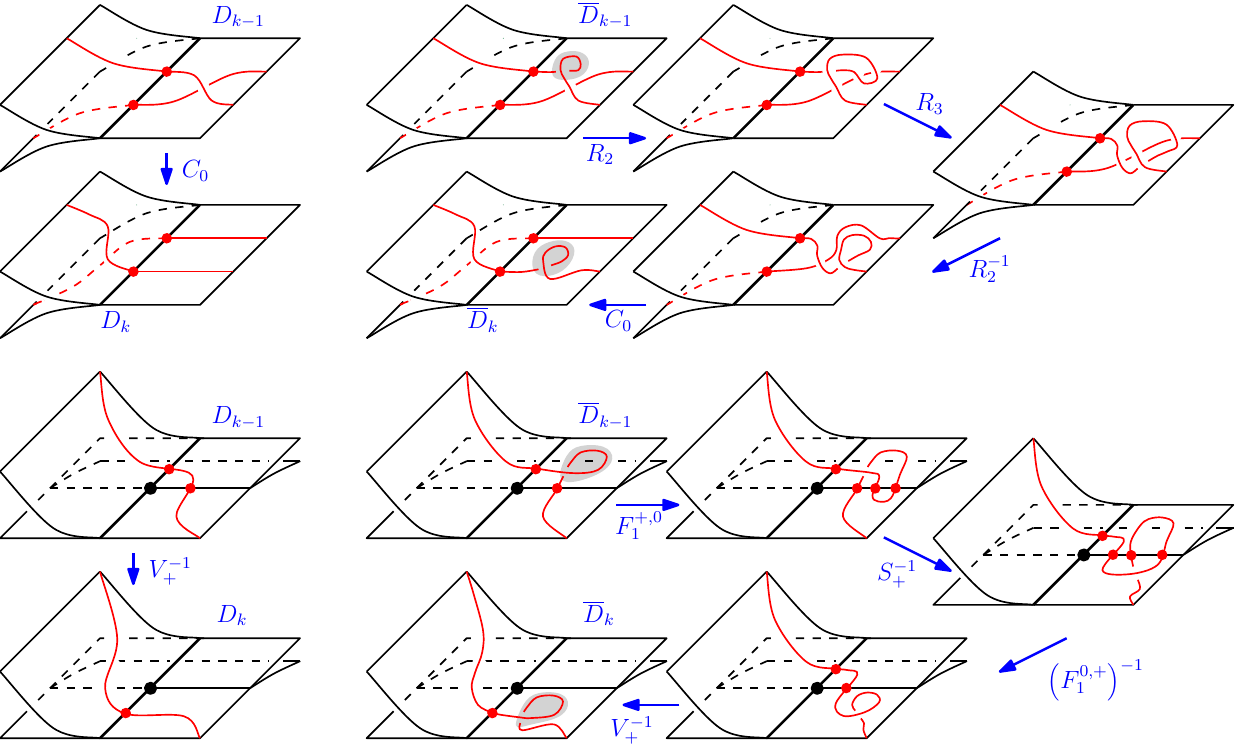}}
{Replacing a move $C_0$ or $V_+^{-1}$.\label{sliding-frozen-curls:fig}}
\end{figure}
\end{itemize}

At the end of this process we have a flow-diagram $D''=\overline{D}_N$ obtained from $D$
by moves as in the statement and coinciding with $D'$ except for the presence of some curls
(we can now disregard the fact that they are frozen), so $B(D')$ and $B(D'')$ are 
isotopic in $M$. In the rest of the proof for the 
sake of simplicity we assume that $L(D'')=L(D')$ is a knot $K$ (in the general case, 
one only needs to repeat the argument for each component of the link).
Moreover we will change $D''$ and $D'$ via moves as in the statement without changing their names.
To begin, using the moves $R_2$, $R_3$, $F_1^{*,*}$, $C_0,\ S_\pm$ we can let the curls slide 
past crossings and intersections with $\Sigma^{(1)}$, until $D''$ and $D'$ coincide outside a small disc
contained in a component of $\Sigma^{(2)}$, where they both consist of an arc with some curls. 
Note that locally, once an orientation is chosen for $K$ there are $4$ types of curls $c_{\pm}^r$ and $c_{\pm}^\ell$, shown in
Fig.~\ref{curltypes:fig}.
\begin{figure}
\figfatta{curltypes}
{\includegraphics[scale=0.6]{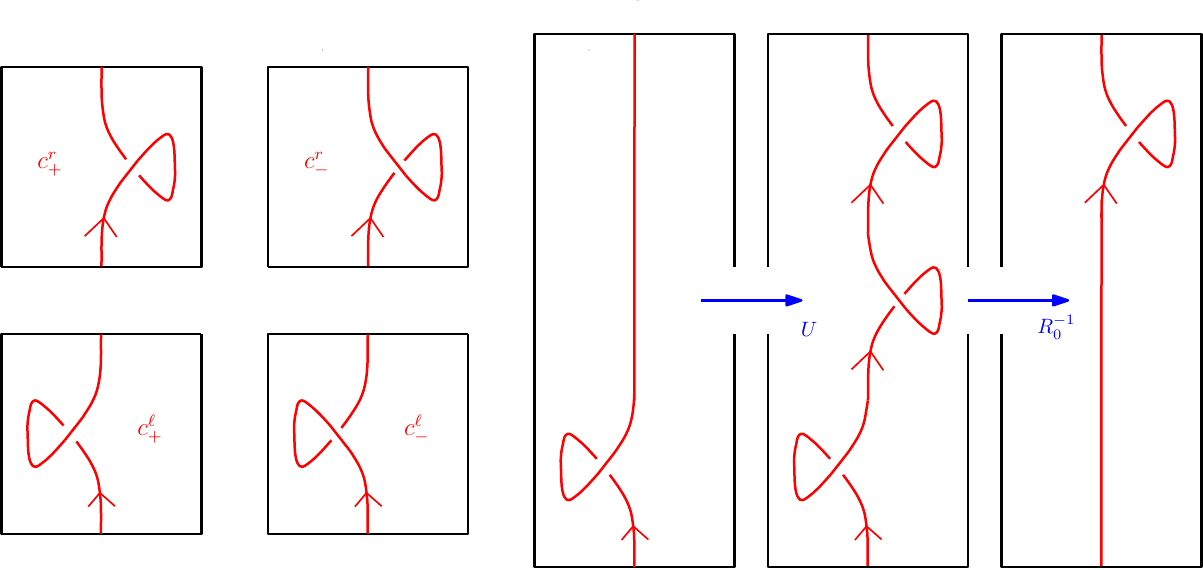}}
{Local curl types, and how to trade a left curl with a right one.\label{curltypes:fig}}
\end{figure}
Moreover, as shown in the same figure, each $c_+^\ell$ can be transformed via $U$ and $R_0$ (whence $R_2$ and $R_3$) into a $c_+^r$, 
and similarly each $c_-^\ell$ into a $c_-^r$. So we can assume that $D''$ and $D'$ only contain curls $c_+^r$ and $c_-^r$, but
if both types occur then two consecutives ones do, and we can cancel them by $U^{-1}$. Eventually, we get $D''$ and $D'$ that
both contain only curls $c_+^r$ or only curls $c_-^r$.

We can now conclude that $D''$ and $D'$ are the same diagram using the fact that $M$ is a homology sphere, because then
the framings of $K$ in $M$ are parameterized by $\matZ$, where $B$ corresponds to $n\in\matZ$ if
$n$ is the algebraic intersection number between one of the components of $\partial R$ and an oriented surface $S$ bounded by $K$,
with $K$ oriented as $\partial S$ and $\partial R$ oriented parallel to its core $K$. 
Adding to a flow-diagram on $\Sigma$ a curl $c_{\pm}^r$ the integer associated to the diagram changes by $\pm1$,
which indeed implies that $D''$ and $D'$ coincide because $B(D')$ and $B(D'')$ are isotopic in $M$.
\end{proof}

\begin{rem}
\emph{The above argument applies in particular to the case where $\Sigma$ is $S^2$ as a flow-spine of $N=S^2\times D^1$,
viewed as $S^3$ minus two open balls, with flow parallel to $D^1$. However on $S^2$ the move $U$ is generated by $R_2$ and $R_3$,
so we recover the known fact framed links in $S^3$ correspond to link diagrams on $S^2$ up to the moves $R_2$ and $R_3$.}
\end{rem}
